\documentclass{amsart}
\usepackage{mathptmx}
\usepackage{helvet}
\usepackage{courier}
\usepackage[T1]{fontenc}
\usepackage[latin9]{inputenc}
\usepackage{graphicx}
\usepackage{amssymb}

\makeatletter
  \theoremstyle{plain}
  \newtheorem{lem}{Lemma}
  \theoremstyle{remark}
  \newtheorem{rem}{Remark}
 \theoremstyle{definition}
 \newtheorem{defn}{Definition}
  \theoremstyle{remark}
  \newtheorem{claim}{Claim}
  \theoremstyle{plain}
  \newtheorem*{algorithm*}{Algorithm}

\title[Computation of the Selberg Zeta Functions on Hecke Triangle Groups]{Computation of the Selberg Zeta Functions on Hecke Triangle Groups and Transferoperators}
\usepackage{multirow}
\usepackage{dcolumn} 
\usepackage{pdflscape}
\usepackage{supertabular}
\newcolumntype{d}[1]{D{.}{.}{#1}}
\newcolumntype{f}[1]{D{+}{+}{#1}}

\newcommand{\PSLR}{PSL_2(\ensuremath{\mathbb{R}})}
\newcommand{\PSLZ}{PSL_2(\ensuremath{\mathbb{Z}})}

\newcommand{\SLR}{SL_2(\ensuremath{\mathbb{R}})}

\newcommand{\sgn}{\text{sgn}}

\newcommand{\Tr}{\  \text{Tr} \ }

\renewcommand{\H}{\ensuremath{\mathcal{H}}}
\newcommand{\Fq}{F_{q}}

\newcommand{\RS}[1]{\ensuremath{\left[ \smash{#1}\vphantom{1^{j}} \right] }}
\newcommand{\DS}[1]{\ensuremath{\left[ \smash{#1}\vphantom{1^{j}} \right]^{*} }}
\newcommand{\Hyp}{\ensuremath{\mathcal{H}_q}}
\newcommand{\Hypp}{\ensuremath{\mathcal{H}^{0}_q}}

\newcommand{\PP}{\ensuremath{\mathcal{P}}}
\newcommand{\PPr}{\ensuremath{\mathcal{P}^{r}}}
\newcommand{\np}{\ensuremath{\kappa}}
\newcommand{\SM}{\ensuremath{\mathcal{T}^{1}\mathcal{M}_{q}}}
\newcommand{\Cq}{\ensuremath{\mathfrak{c}_{q}}}
\newcommand{\Cd}{\ensuremath{\mathfrak{c}_{q}^{*}}}
\newcommand{\nlm}[1]{\ensuremath{ \left\{ #1 \right\} _{\lambda}  } }

\newcommand{\RNCF}{\ensuremath{{\mathcal{A}}_{q} }}

\makeatother

\begin{document}

\title{Computation of Selberg zeta functions on Hecke triangle groups }

\author{Fredrik Strömberg}
\address{Fachbereich Mathematik, AGF, TU Darmstadt,
Schlossgartenstr. 7,
64289 Darmstadt, Germany}
\email{stroemberg@mathematik.tu-darmstadt.de}
\subjclass{Primary:11M36  Secondary: 37C30, 37E05,11F72, 37D40}
\keywords{Selberg zeta function, Symbolic dynamics, Transfer operator, Geodesic flow}
\date{April 30, 2008}

\begin{abstract}
In this paper, a heuristic method to compute the Selberg zeta function
for Hecke triangle groups, $G_{q}$ is described. The algorithm is
based on the transfer operator method and an overview of the relevant
background is given.We give numerical support for the
claim that the method works and can be used to compute the Selberg Zeta
function on $G_{q}$ to any desired precision.
 We also present some numerical results obtained by implementing the algorithm. 

\tableofcontents{}
\end{abstract}
\maketitle

\section{Introduction}

The Selberg zeta function, $Z_{\Gamma}\left(s\right)$, for a co-finite
Fuchsian group $\Gamma$ plays an important role in the spectral theory
or harmonic analysis on the corresponding orbifold $\mathcal{M}=\Gamma\backslash\H$,
a surface with constant negative curvature. Selberg's \cite{MR0088511}
motivation to introduce $Z_{\Gamma}\left(s\right)$ was the similarity
between a trace formula he developed (cf. in particular \cite[p.\ 74]{MR0088511}),
now called the \emph{Selberg trace formula} and Weil's explicit formula
\cite{MR0053152}. The role of the Riemann zeta function $\zeta\left(s\right)$
in the latter is analogous to the role of $Z_{\Gamma}\left(s\right)$
in the former. For a more detailed account of this motivation see
Hejhal \cite{hejhal:76} (in particular sections 4-6). Since then
the Selberg trace formula has been worked out in detail for $\PSLR$
(by e.g.~Hejhal \cite{hejhal:lnm548,hejhal:lnm1001}) and the properties
of $Z_{\Gamma}\left(s\right)$ has been extensively studied in many
other contexts. 

Despite the importance of $Z_{\Gamma}\left(s\right)$ and the fact
that one can obtain an abundance of its properties through the Selberg
trace formula, numerical studies of its behavior inside the critical
strip, $\left|\Re s\right|\le\frac{1}{2}$ have been surprisingly
scarce in the literature. The main reason is of course the fact that
the defining formula does not represent an analytic function in this
domain so one is forced to, one way or the other, analytically continue
this expression.

To the authors knowledge, even for the simple case of the modular
surface, the only successful numerical evaluation of $Z_{\Gamma}$
on the critical line was made by Matthies and Steiner \cite{MR1141108}.
They overcame the difficulty by desymmetrizing the modular surface
$\H/\PSLZ$ with respect to reflection in the imaginary axis and then
restricting their analysis to the odd part, which conveniently avoids
any interference by the continuous part of the spectra. For this system,
corresponding to a billiard with Dirichlet boundary conditions, they
consider a modified Selberg Zeta function, $Z_{-}\left(s\right)$,
which has a Dirichlet series representation which seems to be conditionally
convergent up to $\Re s=\frac{1}{2}$. For convex, co-compact Schottky
groups Guillopé, Lin and Zworski \cite{MR2036371} presented numerical
results for $Z\left(s\right)$ in a large range of $\Im s$. They
use a method based on transfer operators and due to the co-compactness
they are able to evaluate the related Fredholm determinants in a more
or less straight-forward manner in terms of fixed points of the corresponding
maps (cf. e.g. also Jenkinson and Pollicott \cite{MR1902887}). 

In this paper, we also consider an approach to the Selberg zeta function
using a transfer operators. This method is applied to the family of
Fuchsian groups known as Hecke triangle groups, generalizing the modular
group. These groups have finite area but are not co-compact, so the
evaluation of the corresponding Fredholm determinants is more involved.

It can not be stressed too much that at least one of the steps in
our proposed method is not entirely rigorous but rather supported
by heuristic arguments and the entire method is supported by numerical
evidence.

\section{Hyperbolic geometry and Hecke surfaces}

Let $\H=\left\{ z\in\mathbb{C}\,|\,\Im z>0\right\} $ be the hyperbolic
upper half-plane together with the metric given by $ds=\frac{\left|dz\right|}{y}$,
the group of isometries of $\H$ is $\PSLR\cong\SLR/\left\{ \pm I_{2}\right\} $
where $\SLR$ is the group of $2\times2$ real matrices with determinant
$1$ and $I_{2}$ is the $2\times2$ identity matrix. Elements of
$\PSLR$ acts on $\H$ via Möbius transformations. If $g=\left(\begin{smallmatrix}a & b\\
c & d\end{smallmatrix}\right)\in\PSLR$ then $z\mapsto\frac{az+b}{cz+d}$ and we say that we say that $g$
is \emph{elliptic}, \emph{hyperbolic} or \emph{parabolic} depending
on whether $\left|\Tr\, g\right|=\left|a+d\right|<2,$ $>2$ or $=2$.
The same notation applies for fixed points of $g$. A parabolic fixed
point is a degenerate fixed point, belongs to $\partial\H$ and is
usually called a cusp. Elliptic points $z$ appear in pairs, one belongs
to $\H$ and the other one is in the lower half-plane $\overline{\H}$
and its stabilizer subgroup $\Gamma_{z}$ in $\Gamma$ is cyclic of
finite order $m$. Hyperbolic fixed points appear also in pairs with
$x,x^{*}\in\partial\H$, where $x^{*}$ is said to be the conjugate
point of $x$. A geodesics $\gamma$ on $\H$ is either a half-circle
orthogonal to $\mathbb{R}$ or a line parallel to the imaginary axis
and the endpoints of $\gamma$ are denoted by $\gamma_{\pm}\in\partial\H$. 

Let $\pi:\H\rightarrow\mathcal{M}=\Gamma\mathcal{\backslash}\H$ be
the natural projection map, i.e. $\pi\left(z\right)=\Gamma z$ then
$\gamma^{*}=\pi\left(\gamma\right)$ is a closed geodesic on $\mathcal{M}$
if and only if each $\gamma\in\pi^{-1}\left(\gamma^{*}\right)$ has
endpoints which are conjugate hyperbolic fixed points. This gives
a one-to-one correspondence between hyperbolic conjugacy classes in
$\Gamma,$ i.e. the set $\left\{ \left[P\right]\,|\, P\in\Gamma,\,\left|\Tr P\right|>2\right\} $
where $\left[P\right]=\left\{ APA^{-1}\,|\, A\in\Gamma\right\} $.
It is known that any hyperbolic element $P$ can be written as a power
of a primitive hyperbolic element, $P_{0}$, i.e. $P=P_{0}^{m}$ for
some $m\ge1$. We denote this integer by $m\left(P\right)$. In terms
of closed geodesics on $\mathcal{M}$ this means that every closed
geodesic has a minimal length obtained by traversing it once only.
We can now define the Selberg zeta function for $\Gamma$ as \begin{equation}
Z_{\Gamma}\left(s\right)=\prod_{\left[P_{0}\right]\in\Hypp}\prod_{k\ge0}\left(1-\mathcal{N}\left(P_{0}\right)^{-k-s}\right)\label{eq:Zeta}\end{equation}
where $\Hypp$ is the set of primitive hyperbolic conjugacy classes
in $\Gamma$, $P_{0}$ is a representative in this class with $\Tr P_{0}>2$
and the norm of $P,$ $\mathcal{N}\left(P\right)$ is the solution
of $\left|\Tr P\right|=\mathcal{N}^{\frac{1}{2}}+\mathcal{N}^{-\frac{1}{2}}$
with $1<\mathcal{N<\infty}.$ We observe that $\mathcal{N}\left(P_{0}^{n}\right)=\mathcal{N}\left(P_{0}\right)^{n}$
and since the trace is invariant under conjugation $\mathcal{N}$
is constant over conjugacy classes. If $\gamma$ is the geodesic corresponding
to $P$ then the length of $\gamma$ is $l\left(\gamma\right)=\ln\mathcal{N}\left(P\right)$.
For $\Re s>1$ the logarithm of $Z_{\Gamma}\left(s\right)$ can be
written \begin{alignat}{1}
-\ln Z_{\Gamma}\left(s\right) & =-\sum_{k\ge0}\sum_{\left[P_{0}\right]\in\Hypp}\ln\left(1-\mathcal{N}\left(P_{0}\right)^{-k-s}\right)\label{eq:logZ}\\
 & =\sum_{k\ge0}\sum_{\left[P_{0}\right]\in\Hypp}\sum_{n\ge1}\frac{1}{n}\mathcal{N}\left(P_{0}\right)^{-kn-sn}\nonumber \\
 & =\sum_{\left[P_{0}\right]\in\Hypp}\sum_{n\ge1}\frac{1}{n}\mathcal{N}\left(P_{0}\right)^{-sn}\frac{1}{1-\mathcal{N}\left(P_{0}\right)^{-n}}\nonumber \\
 & =\sum_{\left[P_{0}\right]\in\Hypp}\sum_{n\ge1}\frac{1}{n}\frac{\mathcal{N}\left(P_{0}^{n}\right)^{-s}}{1-\mathcal{N}\left(P_{0}^{n}\right)^{-1}}=\sum_{\left[P\right]\in\Hyp}\frac{1}{m\left(P\right)}\frac{\mathcal{N}\left(P\right)^{-s}}{1-\mathcal{N}\left(P\right)^{-1}}.\nonumber \end{alignat}
For an integer $q\ge3$ the \emph{Hecke triangle group} $G_{q}$ is
generated by the maps $S:z\mapsto-\frac{1}{z}$ and $T:z\mapsto z+\lambda_{q}$
where $\lambda_{q}=2\cos\left(\frac{\pi}{q}\right)\in\left[1,2\right)$.
Let $I_{q}=\left[-\frac{\lambda}{2},\frac{\lambda}{2}\right]$. One
can show (cf. e.g. \cite[VII]{MR0164033}) that $G_{q}$ is a Fuchsian
group (discrete subgroup of $\PSLR$) with the only relations $S^{2}=\left(ST\right)^{q}=Id$
and which has $\mathcal{F}_{q}=\left\{ z\in\H\,\large{|}\,\Re z\in I_{q},\,\left|z\right|\ge1\right\} $
as a closed fundamental domain (with sides properly pair-wise identified).
It follows that $G_{q}$ is co-finite, meaning that the \emph{Hecke
triangle surface,} $\mathcal{M}_{q}=G_{q}\backslash\H$, has finite
hyperbolic area. In the following we usually write $\lambda$ for
$\lambda_{q}$, $\Hyp$ and $\Hypp$ denotes the set of hyperbolic
respectively primitive hyperbolic conjugacy classes in $G_{q}$.

\section{Symbolic Coding}

In \cite{symbolic_dynamics} we showed that the geodesic flow on the
unit tangent bundle of $\mathcal{M}_{q}$, $\SM\cong\mathcal{M}_{q}\times S^{1}$
can be coded in terms of regular $\lambda$-fractions (nearest $\lambda$-multiple
continued fractions). If $\nlm{x}=\left\lfloor \frac{x}{\lambda}+\frac{1}{2}\right\rfloor $
is a nearest $\lambda$-multiple function we define $\Fq:I_{q}\rightarrow I_{q}$
by $\Fq\left(0\right)=0$ and $\Fq\left(x\right)=-\frac{1}{x}-\nlm{x}\lambda$
for $x\ne0$. For any number $x\in\mathbb{R}$ we obtain the regular
$\lambda$-fraction of $x$, $\Cq\left(x\right)=\RS{a_{0};a_{1},a_{2},\ldots}$
by first setting $a_{0}=\nlm{x},$ $x_{1}=x-a_{0}\lambda$, and then
recursively set $a_{n}=\nlm{Sx_{n}}$ and $x_{n+1}=\Fq\left(x_{n}\right)$
for $n\ge1$. Note that $x=\lim_{n\rightarrow\infty}T^{a_{0}}ST^{a_{1}}\,\cdots\, ST^{a_{n}}\left(0\right)$.
If $x$ is a cusp of $G_{q}$ this algorithm terminates and we get
a finite $\lambda$-fraction and if $x$ is a hyperbolic fixed point
of $G_{q}$ then it has an eventually periodic $\lambda$-fraction.
It follows that $\Fq$ acts as a left shift map on $\RNCF$, the set
of regular $\lambda$-fractions viewed as a subset of $\mathbb{Z}^{\mathbb{N}}$.
If $a_{0}=0$ we usually omit the leading {}``$a_{0};$'', repetitions
in the $\lambda$-fraction are denoted by powers and infinite repetitions
by overlines. 

In \cite{cf_main} it was shown that $\Fq$ is almost orbit equivalent
to $G_{q},$ that is, two points $x,y\in\mathbb{R}$ are equivalent
under the action of $G_{q}$ if and only if, either they have regular
$\lambda$-fractions with the same tail or $x$ has the same tail
as $r$ and $y$ the same tail as $-r$ (or vice versa). Here $r\in I_{q}$
is a special hyperbolic point which can be given either in terms of
its regular $\lambda$-fraction or explicitly. For even $q$ one has
$r=1-$$\lambda$ and $\Cq\left(r\right)=\RS{\overline{1^{h-1},2}}$
with $h=\frac{q-2}{2}$ and for odd $q$ one has $r=R-\lambda$ where
$R$ is the positive solution of $R^{2}+\left(2-\lambda\right)R-1=1$
and $\Cq\left(r\right)=\RS{\overline{1^{h},2,1^{h-1},2}}$ with $h=\frac{q-3}{2}$. 

Let $\PP$ be the set of all purely periodic regular $G_{q}$-inequivalent
$\lambda$-fractions and set $\PPr=\PP\backslash\left\{ -r\right\} $,
i.e. the set of purely periodic regular $\lambda$-fractions with
tail not equivalent to $-r$. Let $\PPr_{k}$ denote the subset with
minimal period $k\ge1$ and set $\PPr_{0}=\cup_{k\ge1}\PPr_{k}$. 

It was also shown in \cite{symbolic_dynamics} that for the part of
the geodesic flow not disappearing into the cusp there exists a cross
section $\Sigma$ and a first return map $\mathcal{T}:\Sigma\rightarrow\Sigma$
which has as a factor map in the expanding direction the map $\mathcal{T}_{x}:I_{q}\rightarrow I_{q}$
given by powers of the generating map of the nearest $\lambda$-multiple
fractions, $\Fq$. Closed geodesics on $\mathcal{M}_{q}$ correspond
to the orbits of fixed points of $\mathcal{T}$ and it is easy to
verify that these correspond in fact to fixed points of $\Fq$, i.e.
points with purely periodic regular $\lambda$-fractions. It follows
that there is a one-to-one correspondence between $\PPr_{0}$ and
$\Hypp$. 

In practice, if $\Cq\left(x\right)=\RS{\overline{a_{1},\ldots,a_{n}}}$
then $A_{\vec{a}}=ST^{a_{1}}\,\cdots\, ST^{a_{n}}\in G_{q}$ is hyperbolic,
has attractive fixed-point $x$, repelling fixed-point $x^{*}$ and
the geodesic $\gamma\left(x,x^{*}\right)$ is closed. Furthermore
$y=\frac{1}{x^{*}}$ has dual regular $\lambda$-fraction (cf.~\cite{symbolic_dynamics})
$\Cd\left(y\right)=\DS{\overline{a_{n},a_{n-1},\ldots,a_{1}}}$ and
$y\in\left[-R,-r\right]\sgn\left(x\right)$. 

This connection (coding) between primitive hyperbolic conjugacy classes
and periodic orbits of $\Fq$ is precisely what allows us to relate
the Fredholm determinant of the transfer operator for $\Fq$ to the
Selberg zeta function.

\section{The transfer operator}

In this section we will construct the so-called transfer operator
for the map $\Fq$ defined in the previous section.

\subsection{Markov partitions}

There is a particular Markov partition of $I_{q}$ with respect to
$\Fq$ which is important here, namely the one determined by the orbit
$\left\{ \Fq^{j}\left(\pm\frac{\lambda}{2}\right),\, j=1\ldots,\np\right\} $
of the endpoints $\pm\frac{\lambda}{2}$ under $\Fq$. Let $\left\{ \overline{\mathcal{I}}_{j}\right\} _{j\in\mathcal{J}_{\np}}$
be the decomposition of $\left[-\frac{\lambda}{2},\frac{\lambda}{2}\right]$
determined by this orbit with $\mathcal{J}_{\np}=\left\{ 1,2,\ldots,\np,-\np,\ldots,-2,-1\right\} ,$
$\mathcal{I}_{j}=\left[\phi_{j-1},\phi_{j}\right)=-\mathcal{I}_{-j}$
where the order of $\mathcal{O}\left(-\frac{\lambda}{2}\right)=\left\{ \Fq^{j}\phi_{0}\right\} =\left\{ \phi_{j}\right\} _{j=0}^{\np}$
given as $-\frac{\lambda}{2}=\phi_{0}<\phi_{1}<\cdots<\phi_{\np}=0$
and $\phi_{-j}=-\phi_{j}$. If $q$ is even $\np=\frac{q-2}{2}=h$
and $-\frac{\lambda}{2}=\RS{1^{h}}$. If $q$ is odd $\np=\frac{q-3}{2}=2h+1$
and $-\frac{\lambda}{2}=\RS{1^{h},2,1^{h}}$. It is easy to verify
that the closure of the intervals, $\left\{ \overline{\mathcal{I}}_{j}\right\} $
is indeed a Markov partition of $I_{q}$ for $\Fq$. Let $\varphi_{n}\left(y\right)=ST^{n}y=\frac{-1}{n\lambda+y}$
then the most important property of the partition $\left\{ \mathcal{I}_{j}\right\} $
is the fact that if $y\in\mathcal{I}_{j}$ then $\Fq\,^{-1}\left(y\right)=\left\{ \varphi_{n}\left(y\right)\,|\, n\in\mathcal{N}_{j}\right\} $
where $\mathcal{N}_{j}$ is a fixed set of integers depending only
on $j$. It is now easy to show that if $l\ge1$, $i\in\mathcal{J}_{\np}$
and $y\in\mathcal{I}_{i}$ then \[
\left(\Fq\,^{-1}\right)^{l}\left(y\right)=\bigcup_{j\in\mathcal{J}_{\np}}\left\{ \varphi_{n_{l}}\circ\ldots\circ\varphi_{n_{2}}\circ\varphi_{n_{1}}\left(y\right)\,\bigg{|}\,\left(n_{1},\ldots,n_{l}\right)\in\mathcal{N}_{ij}^{l}\right\} \]
where we define $\mathcal{N}_{ij}^{l}:=\left\{ \left(n_{1},n_{2},\ldots,n_{l}\right)\in\mathbb{Z}^{l}\,\bigg{|}\, ST^{n_{l}}ST^{n_{l-1}}\cdots ST^{n_{1}}\mathcal{I}_{i}\subset\mathcal{I}_{j}\right\} .$
Let $\mathbb{Z}_{\ge m}=\left\{ j\in\mathbb{Z}\,|\, j\ge m\right\} $
and for $A\subseteq\mathbb{Z}$ let $-A=\left\{ j\in\mathbb{Z}\,|\,-j\in A\right\} $.
It is shown in \cite{cf_transferoperator} that $\mathcal{N}_{ij}\in\left\{ \left\{ 1\right\} ,\left\{ 2\right\} ,\mathbb{Z}_{\ge2},\mathbb{Z}_{\ge3}\right\} $
for $i,j\in\mathcal{J}_{\np}$, $j\ge0$ and that $\mathcal{N}_{i-j}=-\mathcal{N}_{-ij}$.
It is also shown that the non-empty elements of $\mathcal{N}_{ij}$
(for $j\ge0$) are given by the following expressions: \begin{alignat*}{1}
\mathcal{N}_{1,2h} & =\left\{ 2\right\} ,\mathcal{N}_{1,2h+1}=\mathbb{Z}_{\ge3},\mathcal{N}_{-1,2h}=\mathcal{N}_{-2,2h}=\left\{ 1\right\} ,\\
\mathcal{N}_{i,i-2} & =\mathcal{N}_{-i,2h}=\left\{ 1\right\} ,\mathcal{N}_{i,2h+1}=\mathcal{N}_{-i,2h+1}=\mathbb{Z}_{\ge2},\,3\le i\le2h+1,\\
\mathcal{N}_{2,2h+1} & =\mathcal{N}_{-1,2h+1}=\mathcal{N}_{-2,2h+1}=\mathbb{Z}_{\ge2}\end{alignat*}
if $q$ is odd and \begin{alignat*}{1}
\mathcal{N}_{1,h} & =\mathbb{Z}_{\ge2},\,\,\mathcal{N}_{-1,h}=\mathbb{Z}_{\ge1},\\
\mathcal{N}_{i,i-1} & =\left\{ 1\right\} ,\mathcal{N}_{i,h}=\mathbb{Z}_{\ge2},\,\mathcal{N}_{-i,h}=\mathbb{Z}_{\ge1},\,2\le i\le h\end{alignat*}
if $q$ is even. For example \[
\left(\mathcal{N}_{ij}\right)=\left(\begin{smallmatrix}\mathbb{Z}_{\ge3} & -\mathbb{Z}_{\ge2}\\
\mathbb{Z}_{\ge2} & -\mathbb{Z}_{\ge3}\end{smallmatrix}\right),\,\mbox{for\,}q=3,\,\mbox{and}\,\left(\mathcal{N}_{ij}\right)=\left(\begin{smallmatrix}\mathbb{Z}_{\ge2} & -\mathbb{Z}_{\ge1}\\
\mathbb{Z}_{\ge1} & -\mathbb{Z}_{\ge2}\end{smallmatrix}\right)\,\mbox{for}\, q=4.\]

\subsubsection{Transfer Operator corresponding to $\Fq$}

For any interval $I\subset\mathbb{R}$ let $C\left(I\right)$ denote
the space of continuous real-valued functions on $I$. If $f\in C\left(I_{q}\right)$
the \emph{transfer}, or \emph{generalized Perron-Frobenius operator}
$\mathcal{L}_{\beta}$ corresponding to $\Fq$, acts for real $\beta>\frac{1}{2}$
on $f$ by \begin{eqnarray*}
\mathcal{L}_{\beta}f\left(x\right) & = & \sum_{y\in\Fq^{-1}\left(x\right)}\left|\frac{d}{dx}F_{q}^{-1}\left(x\right)\right|^{\beta}f\left(y\left(x\right)\right)\\
 & = & \sum_{i\in\mathcal{J}_{\np}}\chi_{\mathcal{I}_{i}}\left(x\right)\sum_{n\in\mathcal{N}_{i}}\left|\varphi_{n}'\left(x\right)\right|^{\beta}f\left(\varphi_{n}\left(x\right)\right)\end{eqnarray*}
where $\chi_{\mathcal{I}_{j}}$ is the characteristic function of
$\mathcal{I}_{j}$. It is important to note here, that $\mathcal{L}_{\beta}f\left(x\right)$
is in general not continuous, but only piece-wise continuous. For
this reason consider the action of $\mathcal{L}_{\beta}$ on vector-valued
functions in $\mathcal{C}=\bigoplus_{i\in\mathcal{J}_{\np}}C\left(\mathcal{I}_{i}\right)$.
For $\vec{f}\in\mathcal{C}$ we set $\vec{f}\left(x\right)=f_{i}\left(x\right)$
if $x\in\mathcal{I}_{i}$. We can now write $\mathcal{L}_{\beta}:\mathcal{C}\rightarrow\mathcal{C}$
as \begin{eqnarray*}
\left(\mathcal{L}_{\beta}\vec{f}\right)_{i}\left(x\right) & = & \sum_{j}\sum_{n\in\mathcal{N}_{ij}^{1}}\left|\varphi_{n}'\left(x\right)\right|^{\beta}f_{j}\left(\varphi_{n}\left(x\right)\right),\,\, i\in\mathcal{J}_{\np}\end{eqnarray*}
respectively, for any $l\ge1$ \begin{eqnarray*}
\left(\mathcal{L}_{\beta}^{l}\vec{f}\right)_{i}\left(x\right) & = & \sum_{j}\mathcal{L}_{\beta,ij}^{l}f_{j}\left(x\right),\,\, i\in\mathcal{J}_{\np},\end{eqnarray*}
where\[
\mathcal{L}_{\beta,ij}^{l}f_{j}\left(x\right)=\sum_{\left(n_{1},\ldots,n_{l}\right)\in\mathcal{N}_{ij}^{l}}\left|\left(ST^{n_{1}}\,\cdots\, ST^{n_{^{l}}}\right)'x\right|^{\beta}f_{j}\left(ST^{n_{1}}\,\cdots\, ST^{n_{^{l}}}x\right).\]
It is convenient to use a composition operator $\pi_{\beta}$ related
to the principal series representation of $\PSLR$: Define $\pi_{\beta}\left(A\right)f\left(x\right):=\left|A'\left(x\right)\right|^{\beta}f\left(Ax\right)=\left|cx+d\right|^{-2\beta}f\left(\frac{ax+b}{cx+d}\right)$
for $A\in\PSLR$. Note that $\pi_{\beta}\left(AB\right)=\pi_{\beta}\left(B\right)\pi_{\beta}\left(A\right)$.
With this notation one gets \[
\mathcal{L}_{\beta,ij}^{l}f_{j}\left(x\right)=\sum_{\left(n_{1},\ldots,n_{l}\right)\in\mathcal{N}_{ij}^{l}}\pi_{\beta}\left(ST^{n_{1}}\,\ldots\, ST^{n_{l}}\right)f_{j}\left(x\right).\]
To obtain better spectral properties for the operator $\mathcal{L}_{\beta}$,
we have to restrict its domain of definition even more. For any open
disk $D$ in $\mathbb{C}$ we let $\mathcal{B}\left(D\right)$ be
the Banach space of functions holomorphic in $D$ and continuous on
the closure $\overline{D}$ together with the supremum norm. Let $\left\{ D_{i}\right\} _{i\in\mathcal{J}_{\np}}$
be a set of open disks with diameter which contains an $\epsilon$-neighborhood
of $\mathcal{I}_{i}$, constructed in such a way that for $\left(n_{1},\ldots,n_{l}\right)\in\mathcal{N}_{ij}^{l}$
one has $\varphi_{n_{1}}\circ\varphi_{n_{2}}\circ\cdots\circ\varphi_{n_{l}}\left(\overline{D}_{i}\right)\subset D_{j}$.
That such a choice is possible is shown in \cite{cf_transferoperator}.
Let $\mathcal{B}_{i}=\mathcal{B}\left(D_{i}\right)$ and define the
Banach space $\mathcal{B}=\bigoplus_{i\in\mathcal{J}_{\np}}\mathcal{B}_{i}$
with norm given by $\left\Vert \vec{f}\right\Vert =\max_{j}\left\Vert f_{j}\right\Vert _{\mathcal{B}_{i}}$.
Then we want to consider $\mathcal{L}_{\beta}$ as acting $\mathcal{L}_{\beta}:\mathcal{B}\rightarrow\mathcal{B}$.
For this purpose we also need an analytic extension of $\pi_{\beta}$.
If $A=\left(\begin{smallmatrix}a & b\\
c & d\end{smallmatrix}\right)\in\PSLR$ and $A\left(D\right)\subseteq D$ for some disk $D$ then $\pi_{\beta}\left(A\right):\mathcal{B}\left(D\right)\rightarrow\mathcal{B}\left(D\right)$
is defined for any $\beta\in\mathbb{C}$ by $\pi_{\beta}\left(A\right)f\left(z\right)=\left(\left(cz+d\right)^{-2}\right)^{\beta}f\left(\frac{az+b}{cz+d}\right)$.
Usually we simply write the first factor as $\left(cz+d\right)^{-2\beta},$
but remember that there is a choice of sign involved, i.e. $\left(-cz-d\right)^{-2\beta}=\left(cz+d\right)^{-2\beta}$.
For $l\ge1$ and $\vec{f}\in\mathcal{B}$: \begin{align}
\left(\mathcal{L}_{\beta}^{l}\vec{f}\right)_{i}\left(z\right) & =\sum_{j\in\mathcal{J}_{\np}}\mathcal{L}_{\beta,ij}^{l}f_{j}\left(z\right),\, i\in\mathcal{J}_{\np}\,\mbox{ with}\label{eq:Lbeta^l_def}\\
\mathcal{L}_{\beta,ij}^{l}f_{j}\left(z\right) & =\sum_{\left(n_{1},\ldots,n_{l}\right)\in\mathcal{N}_{ij}^{l}}\pi_{\beta}\left(ST^{n_{1}}\cdots ST^{n_{l}}\right)f_{j}\left(z\right).\nonumber \end{align}
We now have a representation of the operator $\mathcal{L}_{\beta}$
as a $\left(\np+1\right)\times\left(\np+1\right)$ matrix of operators
$\left(\mathcal{L}_{\beta,ij}\right)_{i,j\in\mathcal{J}_{\np}}$ with
$\mathcal{L}_{\beta,ij}:\mathcal{B}_{j}\rightarrow\mathcal{B}_{i}$. 

Next we need some facts from Grothendieck's theory of Fredholm determinants
and nuclear operators on Banach spaces \cite{MR0075539} (Ruelle \cite{MR0420720}
provides more detailed references). The following Lemmas follow from
this theory.

\begin{lem}
\label{lem:atiyah-bott}Let $D$ be any open disk in $\mathbb{C}$
and let $\mathcal{B}\left(D\right)$ be as above. If $\Psi:\mathcal{B}\left(D\right)\rightarrow\mathcal{B}\left(D\right)$
is a simple composition operator $\Psi f\left(z\right)=\psi\left(z\right)f\left(\varphi\left(z\right)\right)$
with $\psi,\varphi$ continuous in $D$ and $\overline{\varphi\left(D\right)}\subset D$.
Then $\varphi$ has an attractive fixed-point $z_{*}\in D$, $\Psi$
is nuclear of order zero and has trace $\Tr_{\mathcal{B}\left(D\right)}\Psi=\frac{\psi(z_{*})}{1-\varphi'(z_{*})}$. 
\end{lem}
The formula for the trace, sometimes referred to as a special case
of the Atiyah-Bott trace formula is easy to verify directly since
the eigenvalues of $\Psi$ are all of the form $\mu_{n}=\psi\left(z_{*}\right)\left(\varphi'\left(z_{*}\right)\right)^{n}$,
$n\ge0$ and $\left|\varphi'\left(z_{*}\right)\right|<1$. 

\begin{lem}
If $\mathcal{L}$ is a nuclear operator of order zero on a Banach
space we can express the Fredholm determinant $\det\left(1-\mathcal{L}\right)$
in two different ways: \[
-\log\det\left(1-\mathcal{L}\right)=\sum_{l=1}^{\infty}\frac{1}{l}\Tr\mathcal{L}^{l}=-\log\prod_{j=1}^{\infty}\left(1-\lambda_{j}\right)\]
where $\left\{ \lambda_{j}\right\} _{j=1}^{\infty}$ are the eigenvalues
of $\mathcal{L}$ (counted with multiplicity). Furthermore, if $\mathcal{L}=\mathcal{L}\left(s\right)$
is a meromorphic function of $s$ then $\det(1-\mathcal{L}\left(s\right))$
is also meromorphic in $s$. 
\end{lem}
\begin{proof}
Cf. e.g. \cite[prop.\ 1, pp.\ 346-347]{MR0075539}.
\end{proof}
\begin{lem}
\label{lem:Let-A-be-hyperbolic}Let $A\in\SLR$ be hyperbolic with
attractive and repelling fixed points $x_{+}$ and $x_{-}$ respectively.
If $D$ is a disk with diameter on $\mathbb{R}$ containing only the
attractive fixed point of $A$ then $\overline{A\left(D\right)}\subset D$. 
\end{lem}
\begin{proof}
This Lemma is easy to verify by conjugating with the map in $\SLR$
which takes $x_{+}$ to $0$, $x_{-}$ to $i\infty$ and $A$ to $z\mapsto l^{2}z$
with $0<l<1$. 
\end{proof}
If $A=\left(\begin{smallmatrix}a & b\\
c & d\end{smallmatrix}\right)$ is hyperbolic with attractive fixed point $x_{+}$ it is easy to
verify that $\mathcal{N}\left(A\right)=j_{A}\left(x_{+}\right)^{2}$
where $j_{A}\left(x\right)=cx+d$. Since $\pi_{\beta}\left(A\right)f\left(x\right)=j_{A}\left(x\right)^{-2\beta}f\left(Ax\right)$
it is easy to see that if $x_{+}\in D$ and $x_{-}\notin D$ then
by Lemma \ref{lem:atiyah-bott} $\pi_{\beta}\left(A\right)$ is nuclear
of order zero and \[
\Tr_{\mathcal{B}\left(D\right)}\pi_{\beta}\left(A\right)=\frac{\mathcal{N}\left(A\right)^{-\beta}}{1-\mathcal{N}\left(A\right)^{-1}}.\]
Let $\vec{n}=\left(n_{1},\ldots,n_{l}\right)\in\mathcal{N}_{jj}^{l}$
and set $A_{\vec{n}}=ST^{n_{1}}ST^{n_{2}}\cdots\, ST^{n_{l}}$. Then
$A_{\vec{n}}\mathcal{\overline{I}}_{j}\subsetneq\mathcal{I}_{j}$
so the attractive fixed point of $A_{\vec{n}}$, $x_{\vec{n}}=\RS{\overline{n_{1},\ldots,n_{l}}}\in\mathcal{I}_{j}$
and by Lemma \ref{lem:Let-A-be-hyperbolic} it is clear that $A_{\vec{n}}\left(\overline{D}_{j}\right)\subset D_{j}$.
This demonstrates that all composition operators showing up in the
operators $\mathcal{L}_{\beta,jj}^{l}$ appearing in the trace of
$\mathcal{L}_{\beta}^{l}$ are nuclear of order zero for $\Re\beta>\frac{1}{2}$.
The arguments in \cite{MR0418168} or \cite{MR1974398} can be generalized
to show that $\mathcal{L}_{\beta}^{l}$ is also of trace class and
nuclear of order zero for $\Re\beta>\frac{1}{2}$. It is clear that
$\Tr_{\mathcal{B}}\mathcal{L}_{\beta}^{l}=\sum_{i\in\mathcal{J}_{\np}}\Tr_{\mathcal{B}_{i}}\mathcal{L}_{\beta,ii}^{l}$
for any $l\ge1$ and it is also not hard to see that any off-diagonal
term, $\mathcal{L}_{\beta,ij}^{l}:\mathcal{B}_{j}\rightarrow\mathcal{B}_{i}$
is a bounded operator. One can now use similar arguments as those
in \cite{MR1789478} to show the following lemma.

\begin{lem}
\emph{If $\Re\beta>\frac{1}{2}$ then $\mathcal{L}_{\beta}$ is nuclear
of order zero and hence of trace class. }
\end{lem}
From the identification of hyperbolic conjugacy classes with purely
periodic $\lambda$-fractions we may now calculate the trace of $\mathcal{L}_{\beta}^{l}$
\begin{eqnarray*}
\Tr_{\mathcal{B}}\mathcal{L}_{\beta}^{l} & = & \sum_{i}\Tr_{\mathcal{B}_{i}}\mathcal{L}_{ii,\beta}^{l}=\sum_{j}\sum_{\vec{n}\in\mathcal{N}_{ii}^{l}}\Tr_{\mathcal{B}_{i}}\pi_{\beta}\left(A_{\vec{n}}\right)\\
 & = & \sum_{\RS{\overline{n_{1},n_{2},\ldots,n_{l}}}\in\PP_{l}}\frac{\mathcal{N}\left(ST^{n_{1}}\cdots ST^{n_{l}}\right)^{-s}}{1-\mathcal{N}\left(ST^{n_{1}}\cdots ST^{n_{l}}\right)^{-1}}\end{eqnarray*}
By the standard Grothendieck theory the Fredholm determinant of $1-\mathcal{L}_{\beta}$
is well-defined and can be calculated by \begin{eqnarray*}
\det\left(1-\mathcal{L}_{\beta}\right) & = & \sum_{l=1}^{\infty}\frac{1}{l}\Tr\mathcal{L}_{\beta}^{l}=\sum_{l=1}^{\infty}\frac{1}{l}\sum_{\RS{\overline{n_{1},n_{2},\ldots,n_{l}}}\in\PP_{l}}\frac{\mathcal{N}\left(A_{\vec{n}}\right)^{-\beta}}{1-\mathcal{N}\left(A_{\vec{n}}\right)^{-1}}\end{eqnarray*}
We now need to study the relation between $\lambda$-fractions and
hyperbolic conjugacy classes in more detail. Let $x=\RS{\overline{n_{1},\ldots,n_{l}}}$
correspond to the hyperbolic $A=A_{\vec{n}}=ST^{n_{1}}\cdots ST^{n_{l}}=P_{0}^{m}$
where $P_{0}$ is a primitive hyperbolic and $m=m\left(A\right)$
then $x=\RS{\overline{n_{1},\ldots,n_{l_{0}}}}$ where $l_{0}$ is
the minimal period and $l_{0}m=l$. Furthermore, all shifts, $\RS{\overline{n_{i},n_{i+1},\ldots,n_{i-1}}}$
for $1\le i\le n_{l_{0}}$ belong to the same conjugacy class $\left[A\right]$
and the norm is also constant over conjugacy classes. Let $\mathcal{B}\left(x\right)=\mathcal{B}\left(D_{j}\right)$
where $x\in D_{j}$, set $A_{r}r=r$ and $\mathcal{K}_{\beta}=\pi_{\beta}\left(A_{r}\right)$.
Then \begin{align*}
 & -\log\det\left(1-\mathcal{L}_{\beta}\right)+\log\det\left(1-\mathcal{K}_{\beta}\right)\\
= & \sum_{l=1}^{\infty}\frac{1}{l}\sum_{\vec{n}\in\mathbb{Z}^{l},\, x=\RS{\overline{n_{1},n_{2},\ldots,n_{l}}}\in\PP}\Tr_{\mathcal{B}_{j\left(x\right)}}\pi_{\beta}\left(A_{\vec{n}}\right)-\sum_{l=1}^{\infty}\frac{1}{l}\Tr_{\mathcal{B}\left(r\right)}\pi_{\beta}\left(A_{r}^{l}\right)\\
= & \sum_{l=1}^{\infty}\frac{1}{l}\sum_{\vec{n}\in\mathbb{Z}^{l},\, x=\RS{\overline{n_{1},n_{2},\ldots,n_{l}}}\in\PPr}\Tr_{\mathcal{B}_{j\left(x\right)}}\pi_{\beta}\left(A_{\vec{n}}\right)\\
= & \sum_{l_{0}=1}^{\infty}\frac{1}{l_{0}}\sum_{\vec{n}_{0}\in\mathbb{Z}^{l_{0}},\, x=\RS{\overline{n_{1},\ldots,n_{l_{0}}}}\in\PPr}\frac{1}{m\left(P\right)}\Tr_{\mathcal{B}_{j\left(x\right)}}\pi_{\beta}\left(A_{\vec{n}_{0}}^{m\left(P\right)}\right)\\
= & \sum_{l_{0}=1}^{\infty}\frac{1}{l_{0}}\sum_{\left[P_{0}\right]\in\Hypp}\frac{l_{0}}{m\left(P\right)}\Tr_{\mathcal{B}_{j\left(\vec{n}_{0}\right)}}\pi_{\beta}\left(P_{0}^{m}\right)\\
= & \sum_{\left[P\right]\in\Hyp}\frac{1}{m\left(P\right)}\frac{\mathcal{N}\left(P\right)^{-s}}{1-\mathcal{N}\left(P\right)^{-1}}\end{align*}
and by comparing with (\ref{eq:logZ}) we see that for $\Re\beta>1$
we have $\ln Z_{q}\left(\beta\right)=\log\det\left(1-\mathcal{L}_{\beta}\right)-\log\det\left(1-\mathcal{K}_{\beta}\right)$
and thus \begin{equation}
Z_{q}\left(\beta\right)=\frac{\det\left(1-\mathcal{L}_{\beta}\right)}{\det\left(1-\mathcal{K}_{\beta}\right)}.\label{eq:Z-eq-detL-detK}\end{equation}
But since the right hand side is in fact meromorphic for $\beta\in\mathbb{C}$
this equation provides an analytic continuation of $Z\left(\beta\right)$
for $\beta\in\mathbb{C}$. Note that the operator $\mathcal{K}_{\beta}$
is a simple composition operator and we can evaluate $\det\left(1-\mathcal{K}_{\beta}\right)$
explicitly. It is easy to show that all eigenvalues of $\mathcal{K}_{\beta}$
are of the form $\mu_{n}=\left(2+\lambda R\right)^{-2\left(n+\beta\right)},$
$n\ge0$ and hence\begin{equation}
\det\left(1-\mathcal{K}_{\beta}\right)=\prod_{n\ge0}\left(1-\mu_{n}\right).\label{eq:det(1-K)}\end{equation}
To evaluate the factor $\det\left(1-\mathcal{L}_{\beta}\right)$ we
use the same identity $\det\left(1-\mathcal{L}_{\beta}\right)=\prod_{n\ge1}\left(1-\lambda_{n}\right)$
where $\left\{ \lambda_{n}\right\} _{n\ge1}$ are the eigenvalues
of $\mathcal{L}_{\beta}$ counted with multiplicity. In the next chapter
we will discuss how to calculate eigenvalues of $\mathcal{L}_{\beta}$.

\begin{rem}
The method to relate $Z_{\Gamma}\left(s\right)$ to Fredholm determinants
of nuclear operators can be extended to any finite dimensional representation
$\chi$ of $\Gamma$. The identity (\ref{eq:Z-eq-detL-detK}) will
hold with $Z_{\Gamma},$ $\mathcal{L}_{\beta}$ and $\mathcal{K}_{\beta}$
replaced by $Z_{\Gamma}^{\chi}\left(s\right)=\prod_{\left[P_{0}\right]}\prod_{k}\det\left(1-\chi\left(P\right)\mathcal{N}\left(P\right)^{-s-k}\right),$
respectively $\mathcal{L}_{\beta}^{\chi}$ and $\mathcal{K}_{\beta}^{\chi}$.
Here $\mathcal{L}_{\beta}^{\chi}$ and $\mathcal{K}_{\beta}^{\chi}$
are obtained by replacing $\pi_{\beta}\left(A\right)$ with $\pi_{\beta}^{\chi}\left(A\right)=\chi\left(A\right)\pi_{\beta}\left(A\right)$
in all formulas. The only problem is to obtain explicit expressions
for the truncated operator $\mathcal{A}_{\beta}^{\left(N\right)}$
which will be introduced in the next section. The algorithm has been
implemented and tested for representations induced by the trivial
representation of the Hecke congruence subgroups $\Gamma_{0}\left(p\right)$
with prime $p$. This allowed us to compute e.g. $Z_{\Gamma_{0}\left(p\right)}\left(s\right)$
for $p=2,5$. 
\end{rem}

\section{Analytic Continuation of $\mathcal{L}_{\beta}$ and Computation}

It turns out that the same analysis which enables us to deduce an
analytic continuation of $\mathcal{L}_{\beta}$ to $\beta\in$$\mathbb{C}$
is also vital to compute the eigenvalues of $\mathcal{L}_{\beta}$. 

We follow the same procedure as in e.g. \cite{MR1059321,MR1691529}
to demonstrate that $\mathcal{L}_{\beta}$ admits a meromorphic extension
to the whole complex plane. First of all we have to change domains
once more. To make some of the calculations easier it is desirable
to work with functions which have power series expansions around zero.
For this purpose we choose open disks $\tilde{D}_{i}\supset D_{i}$
such that $0\in\tilde{D}_{i}$ and $ST^{n}\tilde{D}_{i}\subset\tilde{D}_{j}$
for $n\in\mathcal{N}_{ij}$. That this construction is possible is
shown in \cite{cf_transferoperator}. If $\tilde{\mathcal{B}}_{i}=\mathcal{B}\left(\tilde{D}_{i}\right)$
then $f_{i}\in\tilde{\mathcal{B}}_{i}$ has a power series expansion
centered at zero. 

If $i\in\mathcal{J}_{\np}$ it can be shown that either $\mathcal{N}_{ij}=\left\{ n_{ij}\in\mathbb{Z}\backslash\left\{ 0\right\} \right\} $
or $\mathcal{N}_{ij}=\emptyset$ for $1\le j\le\np-1$, $\mathcal{N}_{i\np}=\left\{ n\in\mathbb{Z}\,|\, n\ge n_{i\np}\right\} $
for some $n_{i\np}\ge1$ and $\mathcal{N}_{i,-j}=-\mathcal{N}_{-i,j}$.
Let $1\le i\le\np$ and consider $\mathcal{L}_{\beta,ij}:\tilde{\mathcal{B}}_{j}\rightarrow\tilde{\mathcal{B}}_{i}.$
Let $f\in\tilde{\mathcal{B}}_{j}$ and $N\ge1$. Taylor's theorem
with remainder gives $f\left(z\right)=\sum_{k=0}^{N}a_{k}z^{k}+R_{N}\left(z\right)$
with $R_{N}\left(z\right)=O\left(\left|z\right|^{N+1}\right)$. Then
\begin{eqnarray*}
\mathcal{L}_{\beta,i\np}f\left(z\right) & = & \sum_{n\in\mathcal{N}_{i\np}}\pi_{\beta}\left(ST^{n}\right)f\left(z\right)=\sum_{n\ge n_{i\np}}\pi_{\beta}\left(ST^{n}\right)f\left(z\right)\\
 & = & \sum_{n\ge n_{i\np}}\left(\frac{1}{z+n\lambda}\right)^{2\beta}f\left(\frac{-1}{z+n\lambda}\right)\\
 & = & \sum_{n\ge n_{i\np}}\left(\frac{1}{z+n\lambda}\right)^{2\beta}\left[\sum_{k=0}^{N}a_{k}\left(\frac{-1}{z+n\lambda}\right)^{k}+R_{N}\left(\frac{1}{z+n\lambda}\right)\right]\\
 & = & \mathcal{A}_{\beta,i\np}^{\left(N\right)}f\left(z\right)+\mathcal{L}_{\beta,i\np}^{(N)}f\left(z\right)\end{eqnarray*}
where $\mathcal{L}_{\beta,i\np}^{\left(N\right)}f\left(z\right)=\mathcal{L}_{\beta,i\np}\left[f\left(z\right)-\sum_{k=1}^{N}a_{k}z^{k}\right]$
is analytic for $\Re\beta>\frac{1-N}{2}$ and in fact nuclear of order
$0$. We also have $\left\Vert \mathcal{L}_{\beta,i\np}^{\left(N\right)}f\left(z\right)\right\Vert _{\infty}\le C\sum_{n\ge n_{i\np}}\left|\frac{1}{n\lambda}\right|^{2\Re\beta+N+1}\rightarrow0$
as $N\rightarrow\infty.$ The operator $\mathcal{A}_{\beta,i\np}^{\left(N\right)}$
on the other hand can be written as \begin{eqnarray*}
\mathcal{A}_{\beta,i\np}^{(N)}f\left(z\right) & = & \sum_{k=0}^{N}\left(-1\right)^{k}a_{k}\sum_{n\ge n_{i\np}}\left(\frac{1}{z+n\lambda}\right)^{2\beta+k}\\
 & = & \sum_{k=0}^{N}\left(-1\right)^{k}a_{k}\lambda^{-2\beta-k}\zeta\left(2\beta+k,\frac{z}{\lambda}+n_{i\np}\right)\end{eqnarray*}
where $\zeta\left(s,z\right)$ is the Hurwitz zeta function. It is
known that for any $z\in\mathbb{C}$ the function $\zeta\left(s,z\right)$
is meromorphic with only one simple pole at $s=1$ with residue $1$.
Hence $\mathcal{A}_{\beta,i\np}^{(N)}$ is of finite rank and meromorphic
in $\beta$ with at most simple poles at the points $\beta_{k}=\frac{-k+1}{2},$
$0\le k\le N$. For the operator corresponding to $\mathcal{N}_{ij}=\left\{ n_{ij}\right\} $
one has \begin{eqnarray*}
\mathcal{L}_{\beta,ij}f\left(z\right) & = & \pi_{\beta}\left(ST^{n_{ij}}\right)f\left(z\right)=\left(z+n_{ij}\lambda\right)^{-2\beta}f\left(\frac{-1}{z+n_{ij}\lambda}\right)\\
 & = & \left(z+n_{ij}\lambda\right)^{-2\beta}\left[\sum_{k=0}^{N}\left(\frac{-1}{z+n_{ij}\lambda}\right)^{k}+R_{N}\left(\frac{-1}{z+n_{ij}\lambda}\right)\right]\\
 & = & \mathcal{A}_{\beta,ij}^{\left(N\right)}f\left(z\right)+\mathcal{L}_{\beta,ij}^{(N)}f\left(z\right)\end{eqnarray*}
where $\mathcal{A}_{\beta,ij}^{\left(N\right)}f\left(z\right)=$$\sum_{k=0}^{N}a_{k}\left(-1\right)^{k}\left(z+n_{ij}\right)^{-k-2\beta}$
and $\mathcal{L}_{\beta,ij}^{\left(N\right)}f\left(z\right)=\left(z+n_{ij}\lambda\right)^{-2\beta}R_{N}\left(\frac{-1}{z+n_{ij}\lambda}\right)=O\left(\left|z+n_{ij}\lambda\right|^{-N-1-2\Re\beta}\right)$.
It is clear that in this case $\mathcal{A}_{\beta,ij}^{\left(N\right)}$
is entire of finite rank and that $\mathcal{L}_{\beta,ij}^{\left(N\right)}$
is entire and nuclear of order $0$. 

Since $N\ge1$ was arbitrary, we conclude that all components $\mathcal{L}_{\beta,jk}$,
have meromorphic continuations to the entire complex plane with at
most simple poles at the points $\beta_{k}=\frac{-k+1}{2},$ $k=1,\ldots$.
The same clearly holds true for the operator $\mathcal{L}_{\beta}$.
Note, that in the determinant $\det\left(1-\mathcal{L}_{\beta}\right)$
poles may well cancel against zeros due to the presence of eigenvalues
equal to one.

\subsection{Computation of $\mathcal{A}_{\beta}^{\left(N\right)}$}

Let $\mathcal{L}_{\beta}=\mathcal{A}_{\beta}^{\left(N\right)}+\mathcal{L}_{\beta}^{\left(N\right)}$
where $\mathcal{A}_{\beta}^{\left(N\right)}$ and $\mathcal{L}_{\beta}^{\left(N\right)}$
have the components $\mathcal{A}_{\beta,jk}^{\left(N\right)}$ respectively
$\mathcal{L}_{\beta,jk}^{\left(N\right)}$ given above. To obtain
a numerical approximation of $\det\left(1-\mathcal{L}_{\beta}\right)$
it is necessary to approximate the spectrum of $\mathcal{L}_{\beta}$.
For this purpose we construct another finite rank approximation of
$\mathcal{A}_{\beta}^{\left(N\right)}$ in terms of a matrix which
is more suitable for computations. 

Let $\mathcal{P}_{N}$ be the space of polynomials of degree less
than or equal to $N$. Then $\mathcal{P}_{N}$ is a subspace of all
$\mathcal{B}_{i}$'s and we let $\Pi_{N}$ denote the projection $\mathcal{B}_{i}\rightarrow\mathcal{P}_{N}$
given by truncation of the power series, i.e. $\Pi_{N}\left(\sum_{k=0}^{\infty}a_{k}z^{k}\right)=\sum_{k=0}^{N}a_{k}z^{k}$.
We will also use $\Pi_{N}$ to denote the projection from $\mathcal{B}=\bigoplus_{i\in\mathcal{J}_{\np}}\mathcal{B}_{i}$
to the space $\bigoplus_{i\in\mathcal{J}_{\np}}\mathcal{P}_{N}$ obtained
by truncating each component. 

We saw, that $\mathcal{A}_{\beta}^{\left(N\right)}$ maps $\bigoplus_{i\in\mathcal{J}_{\np}}\mathcal{P}_{N}=:\mathcal{P}_{N}^{2\np}$
into a space spanned by Hurwitz zeta functions. By applying \emph{$\Pi_{N}$}
to the resulting expression we obtain an operator $\mathcal{A}_{\beta}^{\left(N,N\right)}:\mathbb{C}^{2\np\left(N+1\right)}\rightarrow\mathbb{C}^{2\np\left(N+1\right)}$
which can be represented by a $\np_{N}\times\np_{N}$ complex matrix
where $\np_{N}=2\np\left(N+1\right)$. This construction will now
be explained in detail.

Let $N\ge1$ be a fixed integer, then with the notation as above \begin{eqnarray}
\mathcal{A}_{\beta,i\np}^{\left(N\right)}f\left(z\right) & = & \sum_{k=0}^{N}a_{k}\frac{\left(-1\right)^{k}}{\lambda^{k+2\beta}}\zeta\left(k+2\beta,\frac{z}{\lambda}+n_{i\np}\right)\label{eq:A-eq-Hzeta_sum}\\
 & = & \sum_{k=0}^{N}a_{k}\frac{\left(-1\right)^{k}}{\lambda^{k+2\beta}}\sum_{n=0}^{\infty}\frac{\left(-1\right)^{n}\left(k+2\beta\right)_{n}}{n!\lambda^{n}}\zeta\left(2\beta+k+n,n_{i\np}\right)z^{n}\nonumber \\
 & = & \sum_{n=0}^{\infty}z^{n}\sum_{k=0}^{N}a_{k}\alpha_{i\np,nk}\nonumber \end{eqnarray}
where $\alpha_{i\np,nk}=\alpha_{i\np,nk}\left(\beta\right)=\frac{\left(-1\right)^{k+n}}{n!\lambda^{n+k+2\beta}}\left(k+2\beta\right)_{n}\zeta\left(2\beta+k+n,n_{i\np}\right)$.
For $i\in\mathcal{J}_{\np}$ and $1\le j\le\np-1$ we get \begin{eqnarray*}
\mathcal{A}_{\beta,ij}^{\left(N\right)}f\left(z\right) & = & \sum_{k=0}^{N}a_{k}\left(-1\right)^{k}\left(z+n_{ij}\lambda\right)^{-2\beta-k}\\
 & = & \sum_{k=0}^{N}a_{k}\left(-1\right)^{k}\sum_{n=0}^{\infty}\frac{\left(-1\right)^{n}\left(2\beta+k\right)_{n}}{n!\left(n_{ij}\lambda\right)^{2\beta+k+n}}z^{n}\\
 & = & \sum_{n=0}^{\infty}z^{n}\sum_{k=0}^{N}a_{k}\alpha_{ij,nk}\end{eqnarray*}
where $\alpha_{ij,nk}=\alpha_{ij,nk}\left(\beta\right)=\frac{\left(-1\right)^{n+k}}{n!\lambda^{2\beta+k+n}}\left(2\beta+k\right)_{n}n_{ij}^{-2\beta-k-n}$.
For $n\le-1$ we use a slightly modified definition of $\pi_{\beta}\left(ST^{n}\right),$
namely $\pi_{\beta}\left(ST^{-n}\right)f\left(z\right)=\left(\left(-z+n\lambda\right)^{-2}\right)^{\beta}f\left(\frac{1}{-z+n\lambda}\right)$.
It is then easy to see that $\alpha_{ij,nk}=\left(-1\right)^{n+k}\alpha_{i-j,nk}$
for $1\le j\le\np$. 

By truncating the sum over $n$ at $N$ in the formula for $\mathcal{A}_{\beta,ij}^{\left(N\right)}$
we get operators $\mathcal{A}_{\beta,ij}^{\left(N,N\right)}:\mathcal{P}_{N}\rightarrow\mathcal{P}_{N}$
and $\mathcal{A}_{\beta}^{\left(N,N\right)}=\left(\mathcal{A}_{\beta,ij}^{\left(N,N\right)}\right)_{i,j\in\mathcal{J}_{\np}}$.
Then $\mathcal{A}_{\beta}^{\left(N\right)}=\mathcal{L}_{\beta}\circ\Pi_{N}$
and $\mathcal{A}_{\beta}^{\left(N,N\right)}=\Pi_{N}\circ\mathcal{L}_{\beta}\circ\Pi_{N}$
and by the identification $\mathcal{P}_{N}^{2\np}\cong\mathbb{C}^{2\np\left(N+1\right)}$
it is clear that $\mathcal{A}^{\left(N,N\right)}:\mathcal{P}_{N}\rightarrow\mathcal{P}_{N}$
can be represented by the $\np_{N}\times\np_{N}$ complex matrix \[
A=\left(\alpha_{ij,nk}\right)_{i,j\in\mathcal{J}_{\np},\,0\le n,k\le N}.\]
In the next section we will discuss the relation between the eigenvalues
of $\mathcal{L}_{\beta}$ and those of $\mathcal{A}_{\beta}^{\left(N,N\right)}.$

\subsection{Approximation of the spectrum of $\mathcal{L}_{\beta}$}

Let $\Sigma_{\beta}$ denote the spectrum of the operator $\mathcal{L}_{\beta}$
and $\Sigma_{\beta}^{N}$ the spectrum of $\mathcal{A}_{\beta}^{\left(N,N\right)}$.
Since $\mathcal{L}_{\beta}$ is nuclear, setting $\Pi^{N}=Id-\Pi_{N}$
then $\Pi^{N}$ is bounded and it is easy to verify, that the conditions
of Theorem 2 and Proposition 3 in Baladi and Holschneider \cite{MR1690191}
are satisfied by the approximations $\mathcal{A}_{\beta}^{\left(N,N\right)},$
$N\ge1.$ Hence the following Lemma can be deduced: 

\begin{lem}
\label{lem:eigenval_approx}Let $\vec{f}\in\mathcal{B}$ be an eigenfunction
of $\mathcal{L}_{\beta}$ corresponding to the eigenvalue $\lambda_{\beta}\in\Sigma_{\beta}$
with algebraic multiplicity $d$. Then there exists $N_{0}\ge0$ such
that for all $N\ge N_{0}$ there exist eigenvalues $\lambda_{N,j}\in\Sigma_{\beta}^{N}$
with corresponding eigenfunctions $\vec{f}_{N,j},$ $1\le j\le l$
such that the sum of the algebraic multiplicities of $\lambda_{N,j}$
equals $d$ and \[
\max_{1\le j\le l}\left(\left|\lambda-\lambda_{N,j}\right|,\left\Vert \vec{f}-\vec{f}_{N,j}\right\Vert \right)\le c\left(N\right),\]
where $c\left(N\right)\rightarrow0$ as $N\rightarrow\infty$. 
\end{lem}
\begin{defn}
If $\lambda_{N,j}\in\Sigma_{\beta}^{N}$ is one of the eigenvalues
in Lemma \ref{lem:eigenval_approx} approximating a $\lambda_{\beta}\in\Sigma_{\beta}$
then $\lambda_{N,j}$ is said to be \emph{regular}, otherwise it is
said to be \emph{spurious}. 
\end{defn}
A problem in computing the spectra of $\mathcal{L}_{\beta}$ using
$\mathcal{A}_{\beta}^{\left(N,N\right)}$ is that we do not know a
priori which eigenvalues of $\mathcal{A}_{\beta}^{\left(N,N\right)}$
are regular and which are spurious. A trivial consequence of Lemma
\ref{lem:eigenval_approx} is the following Lemma which gives a necessary
condition for a sequence of eigenvalues $\lambda_{N_{i},j_{i}}\in\Sigma_{\beta}^{N_{i}}$
to be regular.

\begin{lem}
\label{lem:lambda_get_closer}Let $\left\{ \lambda_{N_{i},j_{i}}\right\} _{i\ge1}$
be a sequence of eigenvalues of $\mathcal{A}_{\beta}^{\left(N_{i},N_{i}\right)}$
such that $\lambda_{N_{i},j_{i}}\rightarrow\lambda\in\Sigma_{\beta}$
as $j\rightarrow\infty$. Then for any $\epsilon>0$ there exists
$N_{0}\ge0$ such that $\left|\lambda_{N_{i},j_{i}}-\lambda_{M_{i},j_{i}}\right|<\epsilon$
for all $N_{i},M_{i}\ge N_{0}$. 
\end{lem}
\begin{rem}
Let $\Sigma_{\beta}=\left\{ \lambda_{\beta,n}\right\} _{n\ge1}$ (where
eigenvalues are counted with multiplicity). By Bandtlow-Jenkinson
\cite{MR2350441,B-J_explicit_ev_est} there exist positive constants
$A,c$ such that $\left|\lambda_{\beta,n}\right|\le Ae^{-cn}$. Numerically
we fond that a similar bound seems to hold for the operators $\mathcal{A}_{\beta}^{\left(N,N\right)}.$
It follows that $0$ is a limit point of $\Sigma_{\beta}$ and there
exist many sequences of spurious eigenvalues $\left\{ \lambda_{N_{i},j_{i}}\right\} $
converging to $0$.
\end{rem}
We will now present an algorithm which uses Lemma \ref{lem:lambda_get_closer}
to compute an approximation to $Z_{q}\left(s\right)$, but to put
Lemma \ref{lem:lambda_get_closer} into praxis we first need to make
the following heuristic claims. 

\begin{claim}
\label{cla:There-is-no-spurious-conv}There is no sequence $\lambda_{N_{i},j_{i}}\in\Sigma_{\beta}^{N_{i}},$$i\ge1$
such that $\lambda_{N_{i},j_{i}}\rightarrow\lambda$ unless $\lambda\in\Sigma_{\beta}$
or $\lambda=0$. 
\end{claim}

\begin{claim}
\label{cla:no-gaps}Suppose that $\Sigma_{\beta}^{N}=\left\{ \lambda_{N,i}\right\} _{1\le i\le\np_{N}}$,
$\left|\lambda_{N,n_{1}}\right|\ge\left|\lambda_{N,n_{2}}\right|\ge\cdots\ge\left|\lambda_{N,n_{K}}\right|>0$
are regular eigenvalues with an estimated error $<\epsilon$ and $\left|\lambda_{N,n_{K}}\right|<\epsilon$.
Then there does not exist an eigenvalue $\lambda_{\beta}\in\Sigma_{\beta}$
in the region $\left\{ z\in\mathbb{C}\,|\,\left|z\right|\ge\left|\lambda_{N,n_{K}}\right|,\,\max_{i=1,\ldots,K}\left|z-\lambda_{N,n_{i}}\right|>\epsilon\right\} $.
I.e. the sequence $\left\{ \lambda_{N,n_{i}}\right\} _{1\le i\le K}$
approximates all eigenvalues of $\mathcal{L}_{\beta}$ with absolute
value greater than or equal to $\left|\lambda_{N,n_{K}}\right|$.
\end{claim}
\begin{algorithm*}
\label{alg:algorithm}Let $\delta,\epsilon>0$ and consider $N$ and
$M$ for some $M\ge N+1$. 

Step 1: Compute the two spectra $\Sigma_{\beta}^{N}=\left\{ \lambda_{N,i}\right\} _{1\le i\le\np_{N}}$
and $\Sigma_{\beta}^{M}=\left\{ \lambda_{M,j}\right\} _{1\le j\le\np_{N}}$
(both ordered with non-increasing magnitude and repeated according
to multiplicity) and the relative differences $\delta_{i,j}=\frac{\left|\lambda_{N,i}-\lambda_{M,j}\right|}{\left|\lambda_{N,i}\right|+\left|\lambda_{M,j}\right|}.$ 

Step 2: Let $k=0$ and consider in sequence each $i=1,\ldots,\np_{N}$.
If there exists a $j$ such that $\delta_{i,j}<\delta$ we assume
that $\lambda_{N,i_{k}}$ and $\lambda_{M,j_{k}}$ are approximating
some $\lambda\in\Sigma_{\beta}$ and accordingly increase $k$ by
$1$, set $i_{k}=i$, $j_{k}=j$, $\delta_{k}=\delta_{i_{k}j_{k}}$
and $\tilde{\lambda}_{\beta,k}=\lambda_{M,j_{k}}$. 

Step 3: Let $K$ denote the last value of $k$. Then $\left\{ \smash{\tilde{\lambda}_{\beta,k}}\right\} _{k=1}^{K}$
is an ordered sequence of eigenvalues believed to approximate eigenvalues
of $\mathcal{L}_{\beta}$ and we define \[
\tilde{d}_{N,M}\left(\beta\right)=\prod_{j=1}^{K}\left(1-\tilde{\lambda}_{\beta,k}\right).\]
If $\left|\smash{\tilde{\lambda}_{\beta,K}}\right|>\epsilon$ we increase
$N$ and $M$ and start from Step 1. As will be explained in section
\ref{sub:Error-analysis.} below it might also be necessary to increase
the working precision simultaneously with $N$ in this step. If $\left|\smash{\tilde{\lambda}_{\beta,K}}\right|<\epsilon$
we assume that $\tilde{d}_{N,M}\left(\beta\right)$ approximates $\det\left(1-\mathcal{L}_{\beta}\right)$
and return \[
\tilde{Z}_{q}\left(\beta\right)=\tilde{Z}_{q,N,M}\left(\beta\right)=\tilde{d}_{N,M}\left(\beta\right)\det\left(1-\mathcal{K}_{\beta}\right)^{-1}\]
as a tentative value of $Z_{q}\left(s\right)$ with an assumed error
depending only on $\delta,\epsilon$ and the working precision. The
factor $\det\left(1-\mathcal{K}_{\beta}\right)$ can be computed using
relation (\ref{eq:det(1-K)}) to any desired accuracy. 
\end{algorithm*}

\section{Discussion of the Results }

The numerical method, Algorithm \ref{alg:algorithm}, which is proposed
as a means to evaluate the Selberg zeta function relies on the heuristic
Claims \ref{cla:There-is-no-spurious-conv} and \ref{cla:no-gaps}
above. It is thus clear that no amount of internal {}``consistency
tests'', e.g. stability under change of order of approximation and
variation of the parameters $\epsilon$ and $\delta$, can certify
that the result returned by the algorithm is correct. If Claim \ref{cla:There-is-no-spurious-conv}
is wrong we would obtain extra eigenvalues not associated to $\mathcal{L}_{\beta}$
and on the other hand, if Claim \ref{cla:no-gaps} is incorrect we
might actually miss eigenvalues of comparatively large magnitude.
In both cases we would only be able to approximate $Z_{\Gamma}\left(s\right)$
times some unknown factor. 

The need of an independent test to verify the accuracy of our numerical
results is thus obvious. We propose to use a test relying on the functional
equation of $Z_{\Gamma}\left(s\right)$. The setup will be discussed
in Subsection \ref{sub:The-functional-equation}. 

\begin{rem}
If we were only concerned about zeros of $Z_{\Gamma}$ on the real
axis, i.e. eigenvalues equal to $1$ of $\mathcal{L}_{\beta}$ for
real $\beta$ much more is known about approximation of eigenvalues
and eigenfunctions, cf. e.g.~\cite{MR1679080,MR1830903}.
\end{rem}

\subsection{\label{sub:The-functional-equation}The functional equation for $Z_{q}\left(s\right)$.}

Let $Z_{q}\left(s\right)$ be the Selberg Zeta function for $G_{q}$.
We know \cite[p.\ 499]{hejhal:lnm1001} that \begin{equation}
\frac{Z_{q}\left(1-s\right)}{Z_{q}\left(s\right)}=\varphi_{q}\left(s\right)c\Psi_{q}\left(s\right),\label{eq:functional_equation}\end{equation}
where $\varphi_{q}\left(s\right)$ is the scattering matrix (here
a $1\times1$-matrix), $c=\varphi_{q}\left(\frac{1}{2}\right)=\pm1$
and \begin{eqnarray*}
\Psi_{q}\left(s\right) & = & \frac{\Gamma\left(\frac{3}{2}-s\right)}{\Gamma\left(s+\frac{1}{2}\right)}\mbox{exp}\left(-\frac{q-2}{q}\pi\int_{0}^{s-\frac{1}{2}}t\tan\left(\pi t\right)dt+\right.\\
 &  & +\pi\sum_{k=1}^{q-1}\frac{1}{q\sin\frac{k\pi}{q}}\int_{0}^{s-\frac{1}{2}}\left(\frac{e^{-\frac{2\pi ikt}{m}}}{1+e^{-2\pi it}}+\frac{e^{\frac{2\pi ikt}{m}}}{1+e^{2\pi it}}\right)dt\\
 &  & +\left.\left(1-2s\right)\ln2\right).\end{eqnarray*}
The function $\Psi_{q}\left(s\right)$ can be computed to any desired
degree of accuracy using standard methods of numerical (e.g. Gauss)
quadrature. Evaluation of $\varphi_{q}\left(s\right)$ on the other
hand is more tricky. For $q=3$ there is an explicit formula \cite[p.\ 508]{hejhal:lnm1001}
\begin{equation}
\varphi_{3}\left(s\right)=\sqrt{\pi}\frac{\Gamma\left(s-\frac{1}{2}\right)\zeta\left(2s-1\right)}{\Gamma\left(s\right)\zeta\left(2s\right)}.\label{eq:phi_3-explicit}\end{equation}
For $q\ge4$ the only explicit formula is in terms of a Dirichlet
series with abscissa of absolute convergence equal to $1$, cf. e.g.~\cite[p.\ 569]{hejhal:lnm1001}
and \cite{MR769866}. Note that for $q=4,6$ it might still be possible
to work out explicit formulas for $\varphi_{q}\left(s\right)$ using
the relations between $G_{4},\, G_{6}$ and $\Gamma_{0}\left(2\right),\,\Gamma_{0}\left(3\right)$
respectively. We do not pursue this approach and for all $q\ge4$
we use values of $\varphi_{q}\left(s\right)$ obtained by an algorithm
of Helen Avelin \cite{helen:deform_published}. The main idea of Avelin's
algorithm is that $\varphi_{q}$ occurs in the zeroth Fourier coefficient
of the Eisenstein series $E\left(s,z\right)$ for the group $G_{q}$
and one can use a method based on the $G_{q}-$invariance of $E\left(s,z\right)$
to compute its Fourier coefficients and thus also $\varphi_{q}\left(s\right)$.
This method was first introduced to compute cuspidal Maass waveforms
on Hecke triangle groups by Hejhal \cite{hejhal:99_eigenf,hejhal:calc_of_maass_cusp_forms}.
Later it was generalized to the setting of general subgroups of $\PSLZ$
\cite[Ch.\ 1]{stromberg:thesis} and finally it was generalized to
compute Eisenstein series on Fuchsian groups with one cusp by Avelin
\cite{helen:deform_published}. 

Another application of the functional equation is that we may define
a real-valued function \begin{equation}
\mathcal{Z}_{q}\left(t\right)=Z_{q}\left(\frac{1}{2}+it\right)e^{-i\Theta\left(t\right)}\label{eq:Z_q-real}\end{equation}
where $\Theta\left(t\right)=\frac{1}{2}\mbox{arg}\left(\varphi\left(\frac{1}{2}+it\right)\Psi\left(\frac{1}{2}+it\right)\right)$
and the branch of the argument is chosen so that $\mathcal{Z}_{q}\left(t\right)$
becomes continuous. Note that a single choice of a branch cut is in
general not possible because $\varphi\left(\frac{1}{2}+it\right)\Psi\left(\frac{1}{2}+it\right)$
winds around zero as $t\in\mathbb{R}^{+}$ varies. The advantage of
considering $\mathcal{Z}_{q}\left(t\right)$ is in our case purely
aesthetic, in that we may plot graphs of $\mathcal{Z}_{q}\left(s\right)$. 

It is known \cite[p.\ 498]{hejhal:lnm1001} that $Z_{q}\left(s\right)$
is zero for $s=s_{k}=\frac{1}{2}+ir_{k}$ where $\frac{1}{4}+r_{k}^{2}$
is an eigenvalue of $\Delta$ and at $s=1-\gamma$ where $\varphi_{q}\left(\gamma\right)=0$.
In Figures \ref{fig:ZQ3}-\ref{fig:ZQ5} we plot $\mathcal{Z}_{q}\left(t\right)$
together with blue vertical lines at $t=r_{k}$ and green vertical
lines at $t=\Im\gamma$. The zeros of $Z_{q}\left(s\right)$ on and
off the half-line are clearly visible as zeros and {}``dips''  of
$\mathcal{Z}_{q}\left(t\right)$ at the corresponding points. The
eigenvalues of $\Delta$ were computed by the method of Hejhal indicated
above, see e.g.~\cite{hejhal:99_eigenf,hejhal:calc_of_maass_cusp_forms,stromberg:thesis}
and zeros of $\varphi_{q}\left(s\right)$ were located using Avelin's
algorithm. 

Verification of these zeros as well as the zeros on the real axis
of $Z_{q}\left(s\right)$ (\cite[p.\ 498]{hejhal:lnm1001}) does of
course also lend credibility to our proposed algorithm but since this
verification does not tell us anything about the accuracy for general
$s$ we prefer to concentrate on the error estimate using $\varphi_{q}\left(s\right)$.

\begin{rem}
The actual value of $\varphi_{q}\left(\frac{1}{2}\right)\in\left\{ \pm1\right\} $
can be computed experimentally in two different ways. The straight-forward
way is to use Avelin's method but it is also known (cf.~\cite[p.\ 498]{hejhal:lnm1001})
that $Z_{q}\left(s\right)$ has a simple pole at $s=\frac{1}{2}$
if and only if $\varphi_{q}\left(\frac{1}{2}\right)=-1$. Experiments
performed using both methods indicate that $\varphi_{q}\left(\frac{1}{2}\right)=-1$
for all $q\ge3.$ 

In certain cases one can use the transfer operator to show that $Z_{q}\left(s\right)$
has a singularity at $s=\frac{1}{2}$ by showing that $\mathcal{L}_{\beta}$
has an eigenvalue $\mu_{\beta}\sim\frac{1}{\lambda\left|\beta-\frac{1}{2}\right|}$
in a neighborhood of $\beta=\frac{1}{2}$, but it is not possible
to exclude that this singularity in $\det\left(1-\mathcal{L}_{\beta}\right)$
is canceled by the appearance of an eigenvalue $=1$ for $\mathcal{L}_{\frac{1}{2}}$. 
\end{rem}

\subsection{\label{sub:Error-analysis.}Discussion of data and error analysis.}

The procedure for testing and producing error estimates of the proposed
algorithm to compute $Z_{q}\left(s\right)$ is now clear. Given tentative
values of $Z_{q}\left(s\right)$ and $Z_{q}\left(1-s\right)$, denoted
by $\tilde{Z}_{q}\left(s\right)$ and $\tilde{Z}_{q}\left(s\right)$
we compute the quantity \[
\tilde{\varphi}_{q}\left(s\right)=\tilde{Z}_{q}\left(1-s\right)\tilde{Z}_{q}\left(s\right)^{-1}\Psi\left(s\right)^{-1}\]
and compare this with the value of $c\,\varphi_{q}\left(s\right)$
obtained as described above (in all cases considered here we have
$c=1$). The difference $\left|\varphi_{q}\left(s\right)-\tilde{\varphi}_{q}\left(s\right)\right|$
or relative difference in the neighborhood of a zero of $\varphi\left(s\right)$
gives an estimate of the accuracy of the values $\tilde{Z}_{q}\left(s\right)$
and $\tilde{Z}_{q}\left(1-s\right)$. 

To confirm a value $\tilde{Z}_{q}\left(s\right)$ we thus need also
to compute $\tilde{Z}_{q}\left(1-s\right),$ but on the critical line
with $s=\frac{1}{2}+it$ we have $Z_{q}\left(1-s\right)=\overline{Z_{q}\left(s\right)}$
so we need only compute $\tilde{Z}_{q}\left(s\right)$. 

Using this {}``$\varphi$-test'' we may verify the correctness of
Claims \ref{cla:There-is-no-spurious-conv} and \ref{cla:no-gaps}.
As it turns out, these two claims seems to be correct in theory. In
practice, however, they and the entire algorithm may fail unless the
working precision is increased as necessary. This phenomenon is clearly
visible in Table \ref{tab:CompDN} where we investigate the case $q=3$
and $s=\frac{1}{2}+5i$ using different degrees of approximation $N$
(here $M=N+3$ always) and working precision $WP$. In this table
we list the estimated error, $\left|\tilde{\varphi}_{3}\left(s\right)-\varphi_{3}\left(s\right)\right|,$
the time it took to compute $Z_{3}\left(s\right)$ in seconds, the
number of eigenvalues of $\mathcal{L}_{\beta}$\emph{ }which were
used in the computation, the size of the smallest of those eigenvalues
and the maximum of differences between eigenvalues of $\mathcal{A}^{\left(N,N\right)}$
and $\mathcal{A}^{\left(M,M\right)}$. With working precision of $50$
digits we see that the error decreases as $N=25,$$50$ and $75$.
To further increase $N$ up to $100$ does not improve the accuracy
and increasing $N$ up to $200$ actually results in a worse approximation
than at $N=25$. The reason for this phenomenon is that Claim \ref{cla:There-is-no-spurious-conv}
is violated due to an increasing number of spurious eigenvalues and
in particular there appear spurious eigenvalues which do not vary
fast with $N$. This problem can be overcome by increasing the working
precision, which is demonstrated in the remainder of the table, where
the precision has been increased to $WP=100,$ $150$ and $200$ digits
respectively. To know a priori when the precision has to be increased
one must study more closely the relative differences $\delta_{k}$.
For example, in the case $WP=50$ and $N=200$, the relative differences
for the spurious eigenvalues of a certain magnitude are much larger
than the relative differences of regular eigenvalues. If one sees
such a break from the otherwise almost monotonously increasing $\delta_{k}$
it is a clear sign to increase the working precision. What is not
visible in this table, is that the need for increase in precision
is actually dependent on the matrix size $\np_{N}=2\np\left(N+1\right)$
and not only on $N$.

Table \ref{tab:phi3} contains values of $\tilde{\varphi}_{3}\left(\frac{1}{2}+ni\right),$
$1\le n\le10$, computed using $N=100,\, M=103,$ $\delta=10^{-7}$
and 100 digits working precision. The third column contains the true
error, i.e. $\tilde{\varphi}_{3}$ compared to the explicit formula
(\ref{eq:phi_3-explicit}) for $\varphi_{3}$. One can see that in
this case the true error agrees well with the error estimate in the
fourth column given by the absolute value of the smallest eigenvalue
used in computing $\tilde{d}\left(s\right)$. The fifth column contains
the difference between $\varphi_{3}\left(s\right)$ computed by Avelin's
method (using double precision) and by the explicit formula. 

In Table \ref{tab:phi4} we list values $\varphi_{4}^{A}\left(s\right)$
of $\varphi_{4}\left(\frac{1}{2}+ni\right),$ for $1\le n\le10$ given
by Avelin's algorithm and of $\tilde{\varphi}_{4}\left(s\right)$
by our algorithm using $N=100,\, M=103,$ $\delta=10^{-7}$ and 100
digits precision. The fourth column contains the difference between
these values and the fifth column contains the size of the last eigenvalue
$\tilde{\lambda}_{K}$ used in the evaluation of $\tilde{d}_{NM}\left(s\right)$. 

Comparing the values of $\left|\tilde{\lambda}_{K}\right|$ in Tables
\ref{tab:phi3} and \ref{tab:phi4} we observe that the eigenvalues
of $\mathcal{L}_{\beta}$ for $q=3$ seem to decay more rapidly than
for $q=4$. We would also expect the errors in the tabulated approximations
of $\varphi_{4}\left(s\right)$ to be greater than those of $\varphi_{3}\left(s\right)$
even though we use the same level of precision and approximation.
However, there is no reason to believe that the error in $\varphi_{4}^{A}$
is any worse than in $\varphi_{3}^{A}$. Moreover it is very unlikely
that the values $\varphi_{4}^{A}\left(s\right)$ should agree with
our values $\tilde{\varphi}_{4}\left(s\right)$ to a much larger degree
than the true accuracy, cf.~e.g.~$s=\frac{1}{2}+9i$ where $|\tilde{\lambda}_{K}|=2\cdot10^{-10}$
but $|\varphi_{4}^{A}\left(s\right)-\tilde{\varphi}_{4}\left(s\right)|=2\cdot10^{-15}$.
We conclude that the values of $\left|\tilde{\lambda}_{K}\right|$
do not give an accurate estimate of the true magnitude of the error
for $q=4$ but that they still provide us with an upper bound. 

The final conclusion we can draw from Tables \ref{tab:CompDN}- \ref{tab:phi4}
is that the value of $|\tilde{\lambda}_{K}|$ alone is not enough
to estimate the error in $\tilde{Z}_{q}\left(s\right)$ unless the
working precision is high enough. To completely eliminate the external
test by using $\varphi_{q}$ to confirm values produced by our algorithm
one needs a better understanding of when it is necessary to increase
the working precision.

It is clear, that high precision eigenvalue computations are very
time consuming. To evaluate $Z_{q}\left(s\right)$ to a fixed precision
it is necessary to increase the approximation level $N$ as $\Im s$
grows and this forces a simultaneous increase in the working precision.
Altogether this makes it very time consuming to compute values of
$Z_{q}\left(\frac{1}{2}+it\right)$ for large $t$'s and to reach
even values of $t\approx1000$ for $q=3$ seem to be out of reach
with current methods and hardware. Remember that the size of the matrix
$\np_{N}$ grows with $q$, so similar problems arise when computing
$Z_{q}\left(s\right)$ for large $q$'s. To end this discussion I
would like to give a feeling of the necessary CPU-times. 

To compute $Z_{q}\left(\frac{1}{2}+it\right)$ with an estimated error
of $10^{-5}$ for $q=3$ takes 12 seconds at $t=15$ and 30 minutes
at $t=100$. Decreasing the error to $10^{-9}$ at $t=15$ only increases
the time to 26 seconds. For $q=8$, to compute $Z_{8}\left(\frac{1}{2}+15i\right)$
with an estimated error of $10^{-4}$ takes 5 hours and 57 minutes.
If these figures seem outrageous, remember that for $q=8$ the size
of $\left|\mathcal{J}_{\np}\right|=6$ and in this case $N=190$ so
$\np_{N}=1146$ compared to $\np_{N}=352$ for $q=3$ at $t=100$.
We see that $\frac{1146}{352}\approx3.25$ and $\frac{\mbox{5h57m}}{\mbox{30m}}=\frac{21412}{1800}\approx11.9\approx3.45^{2}$
so the CPU-time increases roughly like the square of the size of the
approximating matrix, which is to be expected from the eigenvalue
computations.

\subsection{Implementation}

The algorithm outlined on p.~\pageref{alg:algorithm} above has been
implemented in Fortran 90 using the ARPREC \cite{arprec} library
for arbitrary precision computations. It was also necessary to write
arbitrary precision Riemann and Hurwitz zeta functions as well as
an arbitrary precision version of the standard linear algebra system
LAPACK. Fortran 90 codes can be made available from the author upon
request.

For the interested reader who is not comfortable with Fortran there
is a version of the algorithm implemented in MuPAD \cite{mupad} and
this version is available from the homepage of the author. The choice
of MuPAD is mostly because of its good multi-precision linear algebra
capabilities. 

Avelin's algorithm is currently only implemented in double precision
FORTRAN 77 hence the MuPAD version only contains the complete error
check using $\varphi_{q}\left(s\right)$ for $q=3$. 

\begin{landscape}\centering
\begin{table}
\caption[]{Demonstrating interplay between working precision and level of approximation in computing $Z_3(s)$}
\label{tab:CompDN}
\begin{tabular}[t]{r|r|d{55}|l|r|r|l|l}
$WP$ & $N$ & \multicolumn{1}{c}{$Z_{3}\left(\frac{1}{2}+5\text{i}\right)$} & 
\multicolumn{1}{c}{$|\tilde{\varphi}-\varphi|$}	 & Time (s) & $K$ & 
\multicolumn{1}{c}{$|\tilde{\lambda}_{K}|$} & $\max_{k} \delta_{k}$  \\
\hline\noalign{\smallskip}
50 &  25 & 1.1954+0.0811i & $1\cdot 10^{-2}$ & 12 & 3 & $5\cdot 10^{-2}$ &  $8\cdot 10^{-9}$\\
   &  50 & 1.192213397499979+            0.074413721696096i & $ 5\cdot 10^{-9} $ & 65 & 12 & $2\cdot 10^{-8}$ &  $5\cdot10^{-8}$\\
   &  75 & 1.192213402687674941183+      0.07441372136992775i   & $7 \cdot 10^{-13}$ & 202& 17 & $3\cdot 10^{-12}$ & $2\cdot 10^{-8}$ \\
   &  100& 1.192213402687674941270+      0.074413721369927750i   & $7\cdot 10^{-13}$ & 839& 19 & $ 3\cdot 10^{-12}$ & $2\cdot 10^{-8}$ \\
   &  200& 1.0165+0.0782i & $3\cdot 10^{-2}$ & 3764 &  19 & $4\cdot 10^{-9}$ & $5\cdot 10^{-8}$ \\
100&100 & 1.192213402686855883038325+   0.0744137213702737315168i& $4\cdot 10^{-19}$ & 599&  26 & $4\cdot 10^{-18}$ & $2\cdot 10^{-10}$\\
   &200 & 1.192213402686855883047193363+& $3\cdot 10^{-25}$& 4324& 34 & $ 1\cdot 10^{-24}$ & $2\cdot 10^{-8}$\\
   &    & 0.0744137213702737317790183373i & & & &\\
   &250 & 1.1837+0.0575i & $3\cdot 10^{-2}$ &10164& 27 & $4\cdot10^{-18}$ & $3\cdot 10^{-8}$ \\
150   &200 &
   1.192213402686855883047193551130117623672+&
   $1\cdot 10^{-36}$ & 6637 & 50 &
    $5\cdot 10^{-36}$ & $1\cdot 10^{-9}$\\
   &    & 0.07441372137027373177901851156938182265i & & & & \\
   &250 & 
   1.192213402686855883047193551130117623672+&
   $1\cdot 10^{-36}$ &    12847 & 50 &
   $5\cdot 10^{-36}$ & $2\cdot 10^{-9}$ \\
   &  & 0.07441372137027373177901851156938182265i & & & & \\
200 & 250 & 1.192213402686855883047193551130117621955253934465021290+
   &$ 6\cdot 10^{-51}$ & 18788 & 68 & $7\cdot 10^{-49}$ & $8\cdot 10^{-8}$ \\
   &   & 0.074413721370273731779018511569381823164472321153746259i &  &  & &\\
\end{tabular}
\end{table}
\end{landscape}

\begin{landscape}\centering
\begin{table}
\centering
\caption{Comparing values of $\varphi_3(s)$ together with different error estimates ($N=100$, $WP=100$, $\delta=10^{-7}$).}
\label{tab:phi3}
\begin{tabular}[t]{d{1}|d{46}|c|c|c|l	|c}
\multicolumn{1}{c}{$n$} & 
\multicolumn{1}{c}{$\tilde{\varphi}_{3}(s)=\tilde{Z}_{3}(1-s)/\tilde{Z}_{3}(s)/\Psi_{3}(s),\quad s=\frac{1}{2}+ni$} & 
\multicolumn{1}{c}{$\left|\tilde{\varphi}_{3}-\varphi_3\right|$} &
\multicolumn{1}{c}{$K$} &
\multicolumn{1}{c}{$\left| \tilde{\lambda}_{K} \right|$} &
\multicolumn{1}{c}{$\max\delta_{k}$} &
\multicolumn{1}{c}{$\left|\varphi^{A}_{3}-\varphi_{3}\right|$} 
\\
\hline\noalign{\smallskip}
1 &

0.523127151694381217718-0.852254646898521675788i&
$1\cdot10^{-20}$ &
27 &
$6\cdot10^{-21}$&
$8\cdot 10^{-9}$ &
$8\cdot10^{-17}$
\\
2 &

0.777709870863430801402-0.628623382289893657301i&
$3\cdot10^{-21}$&
 24 &
$1\cdot10^{-20}$&
$3\cdot 10^{-8}$ &
$1\cdot10^{-16}$
\\
3 &
0.810307536439650550895-0.586004860380103327530i&
$1\cdot10^{-20}$&
 24 &
$9\cdot10^{-21}$&
$2\cdot 10^{-8}$ &
$4\cdot10^{-16}$
\\
4 &
0.784116026298143660937-0.620614258056007668073i&
$4\cdot10^{-20}$&
27 &
$1\cdot10^{-19}$&
$9\cdot 10^{-8}$ &
$6\cdot10^{-16}$
\\	
5 &
0.620614258056007668073-0.709907649199078141317i&
$4\cdot10^{-19}$&
26 &
$4\cdot10^{-18}$&
$2\cdot 10^{-10}$ &
$7\cdot10^{-16}$
\\
6 &
0.473769476721985155012-0.880648898782356035905i&
$4\cdot10^{-20}$&
26 &
$2\cdot10^{-18}$&
$2\cdot 10^{-9}$ &
$1\cdot10^{-15}$
\\
7 &
-0.982666838048427466635+0.185380380299279850669i&
$2\cdot10^{-19}$&
26 &
$1\cdot10^{-18}$&
$9\cdot 10^{-9}$ &
$1\cdot10^{-14}$
\\
8 &

0.947280945444430850195-0.320404136049934947281i&
$3\cdot10^{-19}$&
26 &
$2\cdot10^{-18}$&
$1\cdot 10^{-8}$ &
$2\cdot10^{-15}$
\\
9 &
0.678702274737248216706-0.734413522660418366220i&
$1\cdot10^{-18}$&
26 &
$3\cdot10^{-18}$&
$1\cdot 10^{-8}$ &
$1\cdot10^{-15}$
\\
10 &
-0.063355766687361081600-0.997991005384044865561i&   
$2\cdot10^{-18}$&
26 &
$5\cdot10^{-18}$&
$2\cdot 10^{-8}$ &
$5\cdot10^{-15}$
\\
\end{tabular}
\end{table}
\begin{table}\label{tab:phi4}
\centering
\caption{Comparing values of $\varphi_4(s)$ ($N=100$, $WP=100$, $\delta=10^{-7}$).}
\begin{tabular}[t]{c|d{36}|d{36}|c|c|c|c}
\multicolumn{1}{c}{$n$} & 
\multicolumn{1}{c}{$\varphi^{A}_{4}\left(\frac{1}{2}+ni\right)$} &
\multicolumn{1}{c}{$\tilde{\varphi}_{4}(s)=\tilde{Z}_{4}(1-s)/\tilde{Z}_{4}(s)/\Psi_{4}(s)$} & 
\multicolumn{1}{c}{$\left| \varphi^{A}_{4}-\tilde{\varphi}_{4} \right|$} &
\multicolumn{1}{c}{$K$} &
\multicolumn{1}{c}{$\left| \tilde{\lambda}_{K} \right|$} & 
\multicolumn{1}{c}{$\max\delta_k $} \\
\hline\noalign{\smallskip}
1 &
 -0.2632601861373177 -0.9647248697918721i & 
 -0.2632601861373176 -0.9647248697918723i &
$2\cdot 10^{-16}$ &
27 &
$5\cdot 10^{-14}$ &
$7\cdot 10^{-8}$
\\
2 &
 -0.7021440712594831 -0.7120349030736887i &
 -0.7021440712594827 
-0.7120349030736895i
& $9\cdot 10^{-16}$ &
25 &
$3\cdot 10^{-13}$ &
$5\cdot 10^{-8}$
\\
3 &
-0.9912520623526865  +0.1319823809511951i &
-0.9912520623526863
+0.1319823809511939i
& $1\cdot 10^{-15}$ 
& 24 
& $6\cdot 10^{-12}$ 
& $5\cdot 10^{-9}$
\\
4 &
 0.2148427612152942 + 0.9766486512320531i&
 0.2148427612152920
+0.9766486512320534i
& $2\cdot 10^{-15}$&
26 &
 $3\cdot 10^{-13}$&
 $8\cdot 10^{-8}$
\\
5 &
-0.8749676464424498  -0.4841813892323421i&
-0.8749676464424484
-0.4841813892323440i
& $2\cdot 10^{-15}$
& 24
& $3\cdot 10^{-12}$
& $3\cdot 10^{-8}$
\\
6 &
-0.0732387210885128 +0.9973144387470341i&
-0.0732387210885146
+0.9973144387470377i
& $2\cdot 10^{-15}$
& 24
& $2\cdot 10^{-11}$
& $1\cdot 10^{-8}$
\\
7 &
-0.0299908075389591  -0.9995501745601175i&
-0.0299908075389498
-0.9995501745601176i
& $9\cdot 10^{-15}$
& 24
&  $4\cdot 10^{-11}$
&  $3\cdot 10^{-8}$
\\
8 &
0.8554598916720125 + 0.5178690700751991i&
 0.8554598916721716
+0.5178690700748949i
& $3\cdot 10^{-13}$
& 24
& $1\cdot 10^{-10}$
& $1\cdot 10^{-8}$
\\
9 &
0.7163471899280185  -0.6977440099938026i&
0.7163471899280196
-0.6977440099938012i
& $2\cdot 10^{-15}$
& 24
& $2\cdot 10^{-10}$
& $7\cdot 10^{-8}$
\\
10 &
-0.7358033312663973  -0.6771952877104747i&
-0.7358033312663925
-0.6771952877104797i
& $7\cdot 10^{-15}$
& 24
& $3\cdot 10^{-10}$
& $1\cdot 10^{-7}$
\\
\end{tabular}
\end{table}
\end{landscape}

\begin{figure}

\caption{$\mathcal{Z}_{3}\left(t\right)$\label{fig:ZQ3}}
\includegraphics{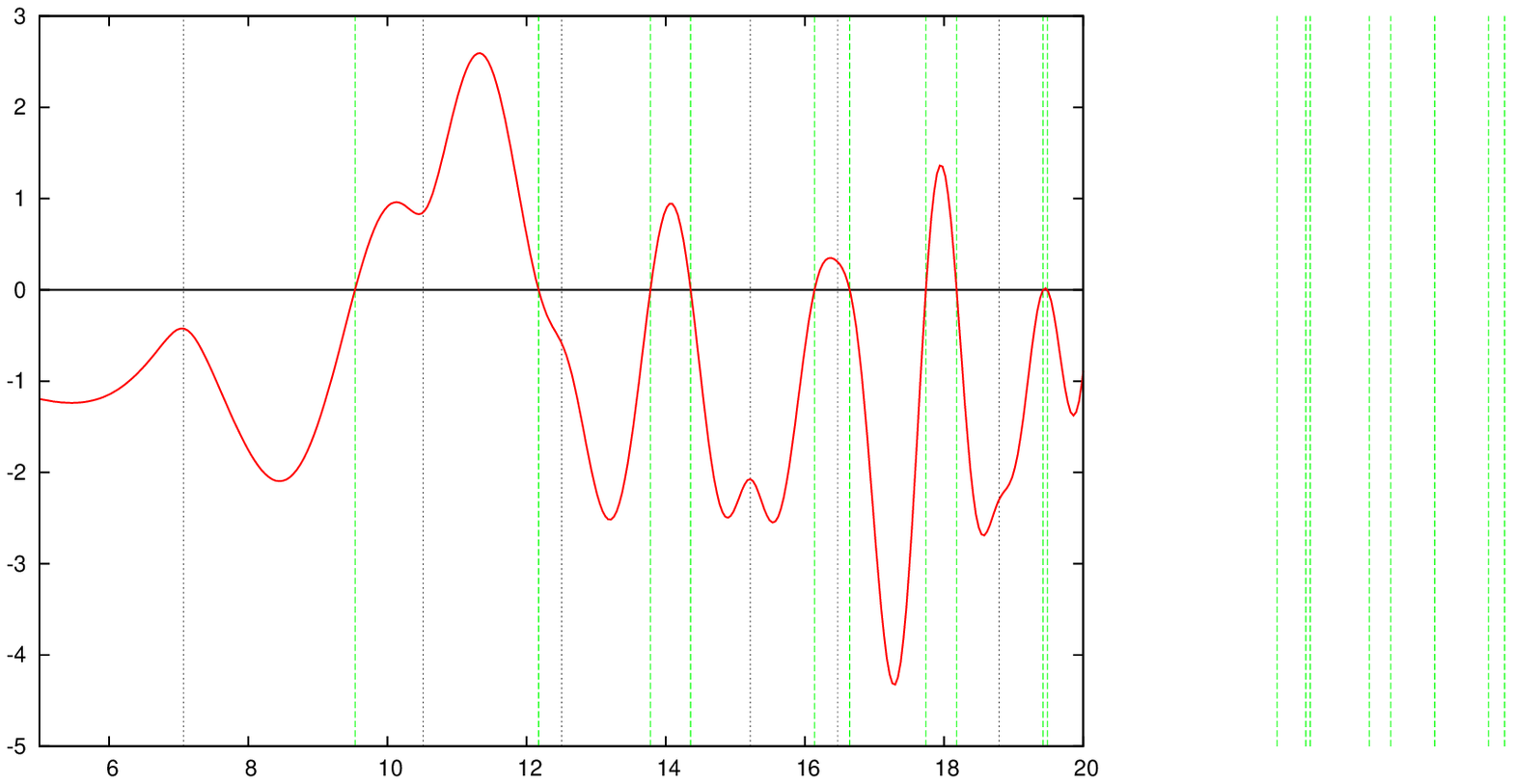}
\end{figure}
\begin{figure}
\caption{$\mathcal{Z}_{4}\left(t\right)$\label{fig:ZQ4}}
\includegraphics{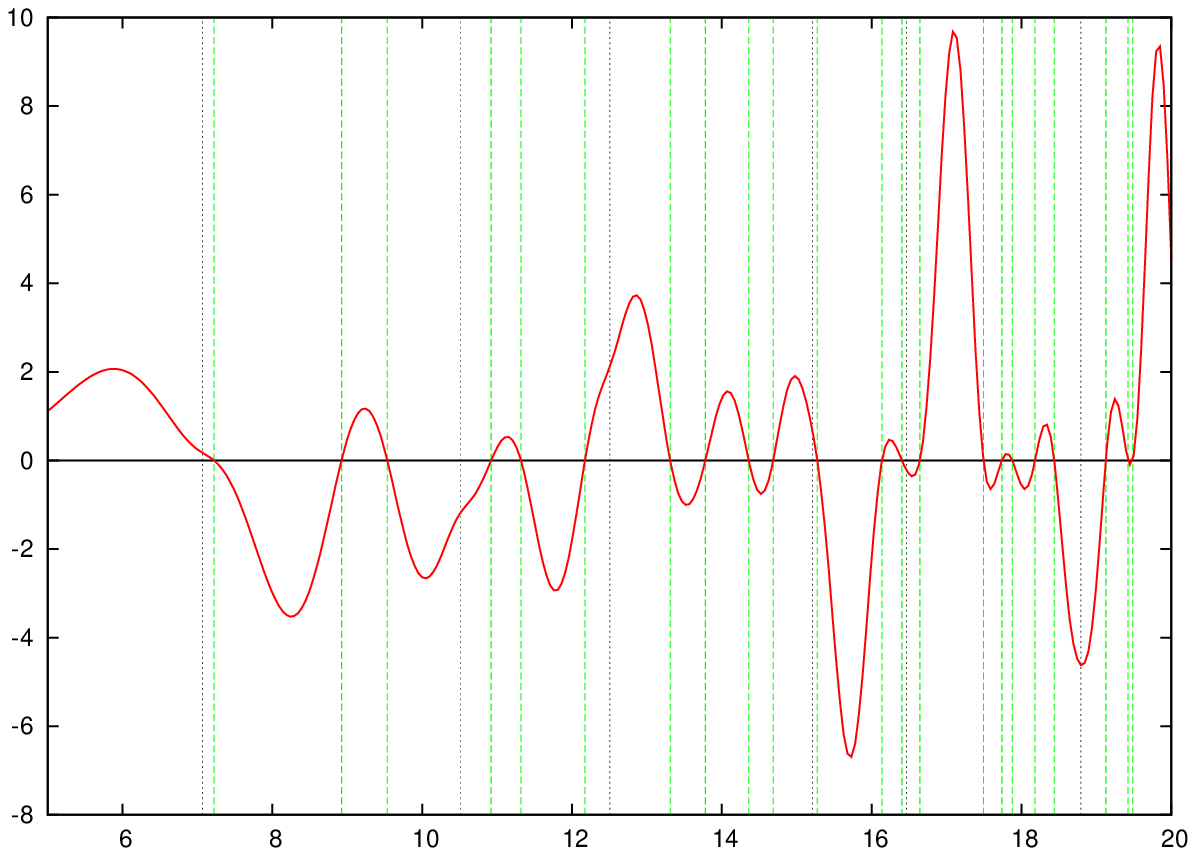}
\end{figure}
\begin{figure}
\caption{$\mathcal{Z}_{5}\left(t\right)$\label{fig:ZQ5}}
\includegraphics{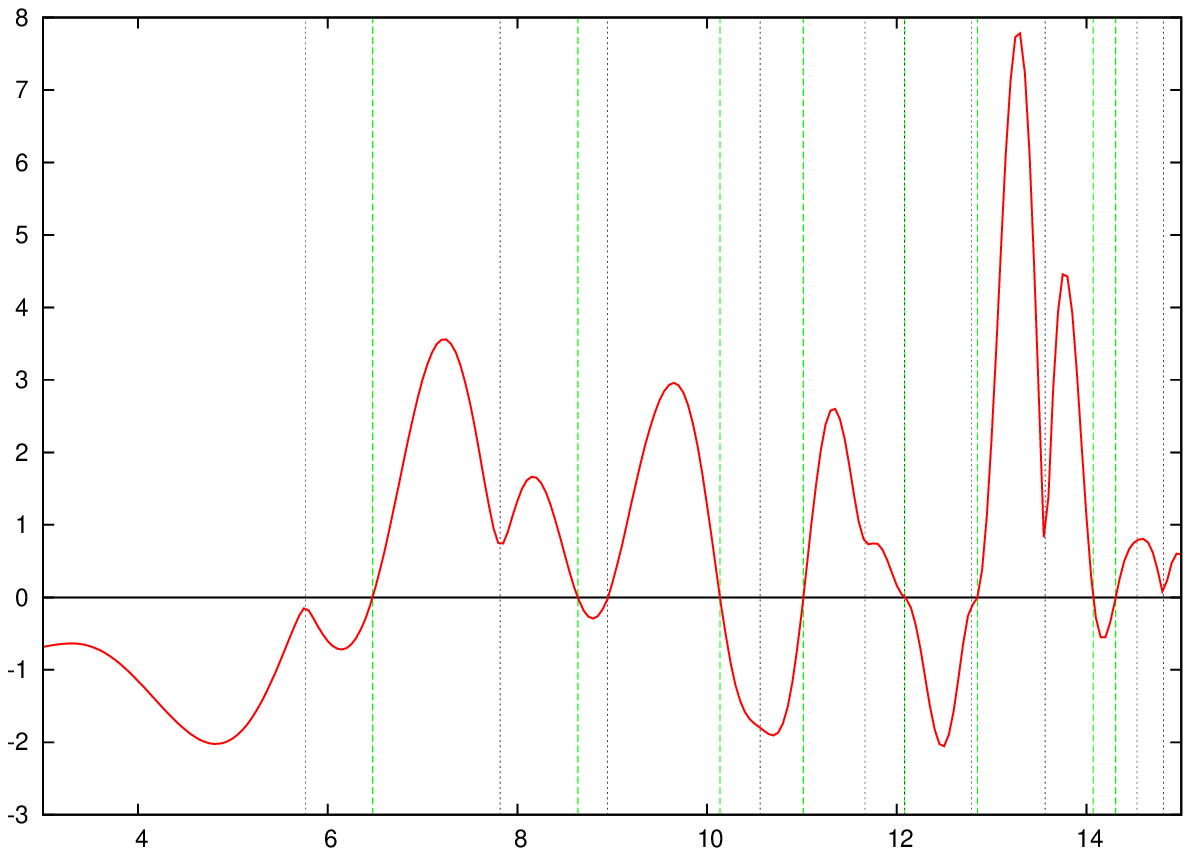}
\end{figure}

\bibliographystyle{amsplain}
\bibliography{/home/fredrik/Documents/matematik/refs}

\end{document}